 \documentclass[11pt]{article}

\usepackage{amssymb,amsmath,amsfonts}
\usepackage{graphicx,color,enumitem}
\usepackage[a4paper, total={6.5in, 8.5in}]{geometry}
\usepackage{mathrsfs}
\usepackage{amsthm} 
\usepackage[dvipsnames]{xcolor}
\usepackage{bm}
\usepackage[round]{natbib}
\usepackage[dvipsnames]{xcolor}
\usepackage{tikz} 
\usepackage{subfigure}
\usetikzlibrary{patterns}
\usepackage{float}
\usepackage{bm}
\usepackage{dsfont}
\RequirePackage[colorlinks,citecolor=blue,urlcolor=blue, linkcolor=red]{hyperref}
\usepackage{bbm}
\usepackage{mathtools}
\usepackage{titlesec}
\usepackage{fancyhdr}
\usepackage{graphicx}
\graphicspath{{figures/}}
\usepackage[]{appendix}
\usepackage{xr}

\newcommand{\E}{{\mathbb E}}
\newcommand{\F}{{\mathbb F}}
\newcommand{\G}{{\mathbb G}}
\newcommand{\N}{{\mathbb N}}
\renewcommand{\P}{{\mathbb P}}
\renewcommand{\S}{{\mathbb S}}
\newcommand{\R}{{\mathbb R}}
\newcommand{\Dcal}{{\mathcal D}}
\newcommand{\Ecal}{{\mathcal E}}
\newcommand{\Fcal}{{\mathcal F}}
\newcommand{\Gcal}{{\mathcal G}}
\newcommand{\Kcal}{{\mathcal K}}
\newcommand{\Pcal}{{\mathcal P}}
\newcommand{\Tcal}{{\mathcal T}}
\newcommand{\Wcal}{{\mathcal W}}
\newcommand{\Xcal}{{\mathcal X}}

\newcommand{\Nu}{{N}}

\newcommand{\fdot}{{\,\cdot\,}}
\DeclareMathOperator{\supp}{supp}
\DeclareMathOperator{\tr}{Tr}
\DeclareMathOperator{\gph}{gph}

\newtheorem{theorem}{Theorem}
\newtheorem{corollary}[theorem]{Corollary}
\newtheorem{lemma}[theorem]{Lemma}
\newtheorem{proposition}[theorem]{Proposition}
\theoremstyle{definition}
\newtheorem{definition}[theorem]{Definition}
\newtheorem{remark}[theorem]{Remark}
\newtheorem{example}[theorem]{Example}

\numberwithin{equation}{section}
\numberwithin{theorem}{section}

\newtheorem{examplex}[theorem]{Example}
\renewenvironment{example}
  {\pushQED{\qed}\examplex}
  {\popQED\endexamplex}

\begin{document}

\title{Existence of probability measure valued jump-diffusions in generalized Wasserstein spaces} 

\author{Martin~Larsson\thanks{Department of Mathematical Sciences, Carnegie Mellon University, Pittsburgh, Pennsylvania 15213, USA, martinl@andrew.cmu.edu.}
\and Sara~Svaluto-Ferro\thanks{Faculty of Mathematics, University of Vienna, Kolingasse 14-16, 1090 Vienna, Austria, sara.svaluto-ferro@uni\-vie.ac.at.\newline
The authors gratefully acknowledge financial
support by the Swiss National Science Foundation (SNF) under grant 
205121$\_$163425. They also thank two anonymous referees for their valuable comments. Sara Svaluto-Ferro gratefully acknowledges financial 
support by the Vienna Science and Technology Fund (WWTF) under grant 
MA16-021.
}}
\date{December 2, 2020}

\maketitle

\abstract{We study existence of probability measure valued jump-diffusions described by martingale problems. We develop a simple device that allows us to embed Wasserstein spaces and other similar spaces of probability measures into locally compact spaces where classical existence theory for martingale problems can be applied. The method allows for general dynamics including drift, diffusion, and possibly infinite-activity jumps. We also develop tools for verifying the required conditions on the generator, including the positive maximum principle and certain continuity and growth conditions. To illustrate the abstract results, we consider large particle systems with mean-field interaction and common noise.}

\


\tableofcontents

\section{Introduction}

In this paper we study existence of probability measure valued jump-diffusions, whose dynamics is specified by means of a martingale problem. Processes taking values in spaces of probability measures play an important role in a number of applied contexts. This includes population genetics (see \cite{E:11} for an overview), stochastic partial differential equations (see e.g.~\cite{FL:92} and \cite{KX:99} among many others), statistical physics (see \cite{HS:87} for an overview), optimal transport (see \cite{V:08} for an overview), and mathematical finance, in particular stochastic optimal control, McKean--Vlasov equations, and mean field games (see e.g.~\cite{CD:17} and the references given there) and stochastic portfolio theory (see e.g.~\cite{F:02}, \cite{FK:09}, and \cite{C:19}).

The mathematical theory of probability measure valued processes has a long history going back to \citet{W:68}, \citet{D:77,D:78}, and \citet{FV:79}. We refrain from a full literature review, but only mention the remarkable collection of St.~Flour lecture notes of \citet{S:91}, \citet{D:93}, and \citet{P:02}, as well as the work of \citet{EK:87,EK:93, EK:05}.

Much of the classical literature on measure valued processes works with the weak topology on the space $M_1(\R^d)$ of all probability measures on $\R^d$ (or some other relevant underlying spaces). There are however other interesting topologies that one can place on spaces of probability measures, that are more appropriate in certain situations. Prominent examples are topologies induced by Wasserstein metrics on the spaces $\Pcal_p(\R^d)$ of probability measures with finite $p$-th moments. A basic reason for considering such stronger topologies is to ensure that the quantities which are naturally associated to the current state of the model are continuous functions on the state space.

The price to pay is that the classical existence theory for martingale problems becomes more difficult to apply. As a result, most proofs of existence of measure valued processes proceed instead via interacting particle systems and a passage to the large-population limit (see for instance the approach presented by \cite{DV:95}). In this paper we prove existence for the limiting system directly, without passing through particle systems.

A key difficulty in using the martingale problem is related to the fact that (say) the Wasserstein space $\Pcal_p(\R^d)$ is not straightforward to compactify. To illustrate this, consider first $M_1(\R^d)$ with the topology of weak convergence. This space fails to be locally compact, and hence does not admit a standard one-point compactification. However, the space $M_1((\R^d)^\Delta)$ of probability measures on the one-point compactification of $\R^d$ \emph{is} compact, and thus fits naturally with classical machinery. This simple procedure does not work for $\Pcal_p(\R^d)$.

In this paper we develop a simple device for embedding $\Pcal_p(\R^d)$, and other similar spaces, into compact spaces where the classical existence theory of martingale problems can be applied. This allows us to establish existence of solutions for martingale problems in spaces of this kind. The operators for which the martingale problem is solved can be very general, including both drift, diffusion, and jumps which can be of infinite activity and even non-summable.

We start in Section~\ref{secmgprob} by reviewing some facts about martingale problems. The core of the paper is Section~\ref{main}, where we state and prove our main abstract result, Theorem~\ref{T_Pw}. There we consider a linear operator $L$ on a carefully chosen domain of test functions. A key assumption on $L$ is, as one would expect, that it satisfy the positive maximum principle. Since $L$ acts on functions of probability measures, it may not be obvious how to verify the positive maximum principle in practice. To remedy this, we develop necessary conditions for optimality, see Theorem~\ref{T_optimality}, that can be used to verify the positive maximum principle for operators of \emph{L\'evy type}, introduced in Section~\ref{S_Levy_type}. This extends results in \citet{CLS:18}. Furthermore, in addition to the positive maximum principle, we impose certain continuity and growth conditions on $L$. In Section~\ref{S_verify_tech} we develop tools to aid the verification of these conditions. Finally, in Sections~\ref{three} and~\ref{IIIsec42}, we discuss some applications that illustrate the scope of the abstract theory. These applications are primarily related to large particle systems with mean-field interaction, where the particles are subject to common noise. In such systems, the limiting empirical distribution of the particles evolves as a probability measure valued stochastic process, whose dynamics can often be described in terms of a martingale problem of the type considered here.

The following notation is used throughout the paper. For a locally compact Polish space $E$, we let $M_+(E)$ denote the Polish space of positive measures on $E$, and $M_1(E)$ the subspace of probability measures. We also write $M(E)=M_+(E)-M_+(E)$ for the space of signed measures on $E$ of bounded variation. These spaces are sometimes considered with the topology of weak convergence (defined using bounded continuous functions and denoted $\mu_n\Rightarrow\mu$) or vague convergence (defined using continuous functions vanishing at infinity). We remark that if $E$ is compact, then $M_1(E)$ is compact and $M_+(E)$ is locally compact. However, if $E$ is  noncompact, $M_1(E)$ is not even locally compact. See for instance Remark~13.14(iii) and Corollary~13.30 in \cite{K:13} for more details. For a Polish space $\Xcal$, we let $C(\Xcal)$ denote the space of all continuous functions $f\colon\Xcal\to \R$. Subscripts $0$ and $c$ indicate that the functions are also vanishing at infinity and have compact support, respectively. If present, a superscript indicates their degree of continuous differentiability.

\section{Martingale problems and the positive maximum principle}\label{secmgprob}

Let $\Xcal$ be a Polish space, $\Dcal\subseteq C(\Xcal)$ a linear subspace, and consider a linear operator
\begin{equation}\label{LDtoCX}
L\colon \Dcal \to C(\Xcal).
\end{equation}
In this paper, $\Xcal$ will be a subset of $M(E)$ for some closed subset $E\subseteq\R^d$, or of $M(E^\Delta)$ where $E^\Delta$ is the one-point compactification of $E$. The topology on $\Xcal$ will however not always be the subspace topology (i.e.\ the topology of weak convergence). Moreover, the functions in $\Dcal$ will usually be defined on a larger subset of $M(E)$ than $\Xcal$, in which case the condition $\Dcal\subseteq C(\Xcal)$ just means that $f|_\Xcal\in C(\Xcal)$ for every $f\in \Dcal$.

\begin{definition}
An $\Xcal$-valued c\`adl\`ag process $X$, defined on some filtered probability space, is called a solution to the martingale problem for $(L,\Dcal,\Xcal)$ with initial condition $\mu\in\Xcal$ if $X_0=\mu$ and
\[
f(X_t) - f(X_0) - \int_0^t Lf(X_s) ds, \quad t\ge0,
\]
is a local martingale for every $f\in\Dcal$.
\end{definition}

It is convenient to allow solutions to the martingale problem to leave the state space. If $\Xcal$ is locally compact, this is formalized via a one-point compactification of $\Xcal$. A similar procedure works more generally. Fix a cemetery state $\dag\notin\Xcal$. Define $\Xcal^\dag=\Xcal\cup\{\dag\}$, and let $\Dcal^\dag$ consist of all $f\colon \Xcal^\dag\to\R$ such that $(f-f(\dag))|_{\Xcal} \in \Dcal$. For every $f\in\Dcal^\dag$, define a function $L^\dag f\colon\Xcal^\dag\to\R$ by $L^\dag f|_{\Xcal} = L((f-f(\dag))|_{\Xcal})$ and $L^\dag f(\dag)=0$. Assume that the given Polish topology on $\Xcal$ can be extended to a Polish topology on $\Xcal^\dag$ in such a way that both $\Dcal^\dag$ and $L^\dag(\Dcal^\dag)$ are contained in $C(\Xcal^\dag)$. For example, this is the case if $\Xcal$ is locally compact, $\Xcal^\dag$ is the one-point compactification, and both $\Dcal$ and $L(\Dcal)$ are contained in $C_0(\Xcal)$. Observe that $\dag$ may or may not be an isolated point. If $\Xcal$ is not locally compact, then the one-point compactification is not available, and other constructions must be used. This situation arises, for instance, when $\Xcal$ is a Wasserstein space of probability measures.

\begin{definition}
A solution $X$ to the martingale problem for $(L^\dag,\Dcal^\dag,\Xcal^\dag)$ with initial condition $\mu\in\Xcal$ is called a possibly killed solution to the martingale problem for $(L,\Dcal,\Xcal)$ with initial condition $\mu$.\footnote{In our terminology, a solution can be killed either by jumping to the cemetery state $\dag$, or by reaching it continuously by means of an ``explosion''.}
\end{definition}

For definiteness, we now suppose that $\Xcal$ is a subset of $M(E)$. We also suppose that for $f\in\Dcal$, both $f$ and $Lf$ are defined on all of $M(E)$. The following classical definition is useful because it can very often be checked in practice.

\begin{definition}
$L$ satisfies the positive maximum principle on $\Xcal$ at $\mu\in M(E)$ if
\[
\text{$f\in\Dcal$ and $f(\mu)=\sup_\Xcal f \ge 0$} \quad\Longrightarrow\quad Lf(\mu)\le0.
\]
If this holds for all $\mu\in\Xcal$, then $L$ is said to satisfy the positive maximum principle on $\Xcal$.
\end{definition}

The positive maximum principle directly implies that $Lf|_\Xcal$ only depends on $f|_\Xcal$ and not on the values $f$ takes outside $\Xcal$. Thus, if $L$ satisfies the positive maximum principle on $\Xcal$, it can be regarded as an operator sending functions on $\Xcal$ to functions on $\Xcal$, 
consistent with \eqref{LDtoCX}. The positive maximum principle is linked to existence of solutions to the martingale problem. The following classical result deals with the locally compact case. The nontrivial part is the forward implication, whose proof can be found, e.g., in \cite[Theorem~4.5.4]{EK:05}.

\begin{theorem} \label{T:existence1}
Assume $\Xcal$ is locally compact, $\Dcal\subseteq C_0(\Xcal)$ is dense, and $L(\Dcal)\subseteq C_0(\Xcal)$. Then $L$ satisfies the positive maximum principle on $\Xcal$ if and only if there exists a possibly killed solution to the martingale problem for $(L,\Dcal,\Xcal)$ for every initial condition $\mu\in\Xcal$. 
\end{theorem}

One is often interested in solutions that are not killed. A general condition for this is that there exist functions $f_n\in\Dcal$ such that $f_n\to1$ and $(Lf_n)^-\to0$ in the bounded pointwise sense. This follows from a slight modification of \cite[Theorem~4.3.8 and Remark~4.5.5]{EK:05}.

Since $M_1(E)$ is compact whenever $E\subset\R^d$ is compact, we obtain the following result as a direct application of Theorem~\ref{T:existence1}.

\begin{corollary}\label{cor1}
Let $\Xcal=M_1(E)$ with $E\subset \R^d$ compact. Assume $\Dcal\subseteq C(\Xcal)$ is a dense subset containing the constant function 1, and $L(\Dcal)\subseteq C(\Xcal)$. Then $L$ satisfies the positive maximum principle on $\Xcal$ if and only if there exists a possibly killed solution to the martingale problem for $(L,\Dcal,\Xcal)$ for every initial condition $\mu\in \Xcal$. If additionally $L1=0$, then every such solution $X$ satisfies $X_t\in M_1(E)$ for all $t\ge0$, and is thus a solution to the martingale problem for $(L,\Dcal,\Xcal)$.
\end{corollary}

\section{Main result}\label{main}

Let $w\colon\R^d\to[1,\infty)$ be a $C^\infty$ function such that
\begin{equation}\label{eq_wtoinfty}
\lim_{|x|\to\infty}w(x)=\infty,
\end{equation}
and fix a closed subset $E\subseteq\R^d$. Define the set of probability measures on $E$ with finite $w$-moment,
\[
\Pcal_w = \Pcal_w(E) = \{\mu\in M_1(E)\colon \langle w,\mu\rangle<\infty\},
\]
topologized by the following notion of convergence: $\mu_n\to\mu$ if and only if $\mu_n\Rightarrow\mu$ and $\langle w,\mu_n\rangle\to\langle w,\mu\rangle$. This turns $\Pcal_w$ into a Polish space. A possible choice of metric is
\begin{equation}\label{eq_dw_metric}
d_w(\mu_1,\mu_2) =d(w\mu_1, w\mu_2),
\end{equation}
where $d(\fdot,\fdot)$ is the Prokhorov metric on $M_+(E)$, and the measures $w\mu_i$ are given by $(w\mu_i)(dx)=w(x)\mu_i(dx)$. The Prokhorov metric is discussed in detail in Section~3.1 of \cite{EK:05}. See also the discussion after Example~A.42 in \cite{FS:04}.

\begin{example}
If $w(x)=|x|^p$ outside some ball around the origin, then $\Pcal_w$ is the set of probability measures on $E$ with finite $p$-th moments, and \eqref{eq_dw_metric} generates the same topology as the Wasserstein $p$-distance $\Wcal_p$. 
\end{example}

We will use the following class of test functions:
\begin{equation}\label{Dw}
\Dcal_w = \text{algebra generated by all $\mu\mapsto\langle \varphi,\mu\rangle e^{-\langle w,\mu\rangle}$ with $\varphi\in C^\infty_c(\R^d)$}.\footnote{That is, $\Dcal_w$ consists of all sums of products of functions of the given form $\langle \varphi,\mu\rangle e^{-\langle w,\mu\rangle}$. It does not contain the constant function $1$.}
\end{equation}
With the convention $\exp(-\langle w,\mu\rangle)=0$ if $\langle w,|\mu|\rangle=\infty$, functions in $\Dcal_w$ can be evaluated at any $\mu\in M(E)$. We will obtain possibly killed solutions to the martingale problem for $(L,\Dcal_w,\Pcal_w)$, where $L$ is an operator satisfying suitable assumptions. In order to do so, fix a cemetery state $\dag$ and define $\Pcal_w^\dag=\Pcal_w\cup\{\dag\}$.\footnote{We may take any $\dag\notin M(E^\Delta)$, where $E^\Delta$ is the one-point compactification of $E$.} The topology is extended to $\Pcal_w^\dag$ by declaring that a sequence of measures $\mu_n\in\Pcal_w$ converges to $\dag$ if $\langle w,\mu_n\rangle\to\infty$. Thus $\lim_{\mu\to\dag}f(\mu)=0$ for any $f\in\Dcal_w$, so that $\Dcal_w^\dag$ as defined in Section~\ref{secmgprob} is indeed contained in $C(\Pcal_w^\dag)$. If one assumes that $\lim_{\mu\to\dag}Lf(\mu)=0$ for every $f\in\Dcal_w$, which we shall, it follows that $L^\dag(\Dcal^\dag)\subseteq C(\Pcal_w^\dag)$ as well, where $L^\dag$ is defined as in Section~\ref{secmgprob}. This allows us to speak about possibly killed solutions to the martingale problem.

If $E$ is compact, then $\Pcal_w=M_1(E)$ is also compact, and Corollary~\ref{cor1} yields a satisfactory existence theory for the martingale problem. From now on we consider the opposite situation, and assume that
\[
\text{$E$ is not compact.}
\]
In this case $\Pcal_w$ is not even locally compact, and the classical results are not directly applicable. Instead, we will embed $\Pcal_w$ into a space that is locally compact, where Theorem~\ref{T:existence1} can be applied. To describe this embedding, let $E^\Delta=E\cup\{\Delta\}$ be the one-point compactification of $E$, for some $\Delta\notin E$. The space $M_+(E^\Delta)$ is equipped with the weak topology. Define a map
\begin{equation}\label{eq_T_embedding}
T\colon \Pcal_w \to M_+(E^\Delta), \quad T(\mu)(dx) = w(x)\mu(dx \cap E),
\end{equation}
which is a topological embedding of $\Pcal_w$ into $M_+(E^\Delta)$. Recall that $w\geq1$ and observe that
\[
T(\Pcal_w) = \{\nu\in M_+(E^\Delta) \colon \langle w^{-1}, \nu\rangle = 1,\ \nu(\{\Delta\})=0\},
\]
where $w^{-1}(x)=1/w(x)$, which is well defined everywhere on $E^\Delta$ with the convention $w^{-1}(\Delta)=\lim_{|x|\to\infty}w^{-1}(x)$.
Let $\Xcal$ denote the weak closure of $T(\Pcal_w)$; this will serve as state space for an auxiliary martingale problem. Since $\Xcal$ is a closed subset of the locally compact Polish space $M_+(E^\Delta)$, it is itself locally compact Polish. This places us in the framework of Theorem~\ref{T:existence1}. Note that we have the explicit description
\begin{equation}\label{Xdescr}
\Xcal = \{\nu\in M_+(E^\Delta) \colon \langle w^{-1}, \nu\rangle = 1\}.
\end{equation}
 In particular, a measure $\nu\in\Xcal$ lies in $T(\Pcal_w)$ if and only if it does not charge $\Delta$. 

Using $T$, any martingale problem with state space $\Pcal_w$ and operator $f\mapsto Lf$ can be regarded as a martingale problem with state space $T(\Pcal_w)$ and operator $\widetilde f\mapsto L(\widetilde f\circ T)\circ T^{-1}$. Our strategy is to extend this to a martingale problem with state space $\Xcal$ and, then, show that the solution does not charge $\Delta$ and thus actually lies in $T(\Pcal_w)$. This gives a solution to the original martingale problem.

These steps depend in a somewhat delicate way on the particular choice \eqref{Dw} of test functions. In particular, in order to apply Theorem~\ref{T:existence1}, the function $f\circ T^{-1}$ obtained by pushing forward a function $f\in\Dcal_w$ using $T$ needs to be extendible to a function in $C_0(\Xcal)$. This is captured by the following definition.

\begin{definition}\label{def_C}
A function $f\colon\Pcal_w\to\R$ is of {\em $C_0$ type} if $f\circ T^{-1}\colon T(\Pcal_w)\to\R$ extends to a $C_0$ function on $\Xcal$. This extension is again denoted by $f\circ T^{-1}$.
\end{definition}

It is clear that sums and products of functions of $C_0$ type are again of $C_0$ type; these functions thus form an algebra. Since $\nu\mapsto\langle\varphi,T^{-1}(\nu)\rangle=\langle\varphi w^{-1},\nu\rangle$ is continuous on $\Xcal$ for any $\varphi\in C^\infty_c(\R^d)$, the product of $\langle\varphi,\mu\rangle$ and a function of $C_0$ type is again of $C_0$ type. Also, $\mu\mapsto e^{-\langle w,\mu\rangle}$ is certainly of $C_0$ type. We deduce in particular that every $f\in\Dcal_w$ is of $C_0$ type.

\begin{example}
Suppose $E=\R$ and let $f(\mu)=\langle \varphi,\mu\rangle e^{-\langle w,\mu\rangle}$ for some $\varphi\in C(\R)$. When is $f$ of $C_0$ type? Set $\mu_n=w(n)^{-1}\delta_n + (1-w(n)^{-1})\delta_1 \in \Pcal_w$. Then $\mu_n$ does not converge to any element of $\Pcal_w$, but $\nu_n=T(\mu_n)=\delta_n+w(1)(1-w(n)^{-1})\delta_1$ converges to $\delta_\Delta+w(1)\delta_1$ in $M_+(\R^\Delta)$. On the other hand, $\lim_n f\circ T^{-1}(\nu_n)=(\lim_n \varphi(n)/w(n) + \varphi(1)) e^{-1-w(1)}$ only exists if $\varphi(n)/w(n)$ has a finite limit. By considering similar sequences $\mu_n$, one sees that $f$ is of $C_0$ type if and only if $\varphi(x)/w(x)$ has a finite limit as $x\to\Delta$.
\end{example}

The following is the main result of this paper. To state it, we define the compact subset $\Xcal_c=\{\nu\in\Xcal\colon \langle 1,\nu\rangle\le c\}$ for any constant $c\ge1$. The meaning of its conditions, and examples of how they can be verified, are discussed in later sections.

\begin{theorem}\label{T_Pw}
Consider a linear operator $L\colon \Dcal_w\to C(\Pcal_w)$, and assume the following conditions are satisfied:
\begin{enumerate}
\item\label{T_Pw_1} $L$ satisfies the positive maximum principle on $\Pcal_w$,
\item\label{T_Pw_2_new} $Lf$ is of $C_0$ type for every $f\in\Dcal_w$,
\item\label{T_Pw_3_new} for every constant $c\ge1$, there exist a function $\widetilde f\colon\Xcal\to\R$ and pairs $(\widetilde f_m,\widetilde g_m)$ in the bp-closure of the restricted graph
\begin{equation}\label{T_Pw_3_eq1}
\{(\widetilde f, \widetilde g)\in C_0(\Xcal)\times C(\Xcal_c)\colon \widetilde f\circ T\in \Dcal_w,\, \widetilde g= L(\widetilde f\circ T) \circ T^{-1}|_{\Xcal_c}\}
\end{equation}
such that $(\widetilde f_m,\widetilde g_m^+)$ are uniformly bounded in $m$, and
\begin{enumerate}
\item\label{cond_a_new} $\widetilde f_m\to\widetilde f$ pointwise, $\widetilde f\ge0$, and $\widetilde f(\nu)=0$ for $\nu\in\Xcal_c$ if and only if $\nu(\{\Delta\})=0$,
\item\label{cond_b_new} $\limsup_{m\to\infty} \widetilde g_m^+ \le c' \widetilde f|_{\Xcal_c}$ pointwise for some constant $c'$.
\end{enumerate}
\end{enumerate}
Then there exists a possibly killed solution to the martingale problem for $(L,\Dcal_w,\Pcal_w)$ for every initial condition $\mu\in\Pcal_w$. Furthermore, assume that
\begin{enumerate}[resume]
\item\label{T_Pw_6}
there exist pairs $(f_n,g_n)$ in the bp-closure of the graph $\{(f,Lf)\colon f\in\Dcal_w\}$ of $L$ such that $(f_n,g_n^-)\to (1,0)$ in the bounded pointwise sense.
\end{enumerate}
Then every possibly killed solution $X$ to the martingale problem for $(L,\Dcal_w,\Pcal_w)$ satisfies $X_t\in \Pcal_w$ for all $t\ge0$, and is thus a solution to the martingale problem for $(L,\Dcal_w,\Pcal_w)$.
\end{theorem}

\begin{remark}\label{rem2}
While our focus is the case where \eqref{eq_wtoinfty} holds, one could also take $w\equiv1$. In this case $\Pcal_w=M_1(E)$ is the set of all probability measures on $E$ with the topology of weak convergence. A slight modification of our main result holds also for this case. Specifically, letting $\Dcal_w$ denote the algebra generated by $\langle\varphi,\mu\rangle$ with $\varphi\in\R+C^\infty_c(\R^d)$, Theorem~\ref{T_Pw} remains true as stated. Note that condition \ref{T_Pw_3_new} only needs to be verified for $c=1$. The proof remains unchanged, apart from slightly different arguments in Lemma~\ref{L_Dprops}\ref{L_Dprops:3}--\ref{L_Dprops:4} below.
\end{remark}

The rest of this section is devoted to the proof of Theorem~\ref{T_Pw}, so we now assume that its conditions are satisfied. As discussed above, the proof uses the embedding $T$ in \eqref{eq_T_embedding} to transform the original martingale problem into an auxiliary martingale problem on the state space $\Xcal$ in \eqref{Xdescr}. The domain of test functions for the auxiliary martingale problem is
\begin{equation} \label{eq_T_Dw}
\Dcal = \text{algebra generated by all $\nu\mapsto\langle \varphi,\nu\rangle e^{-\langle 1,\nu\rangle}$ with $\varphi\in C^\infty_c(\R^d)$}.
\end{equation}
The elements of $\Dcal$ can be evaluated at any $\nu\in M(E^\Delta)$, with the conventions $\varphi(\Delta)=0$ and $1(\Delta)=1$. Note also that $\Dcal\subset C_0(\Xcal)$, and that its elements $\widetilde f$ satisfy $\widetilde f\circ T\in\Dcal_w$. Due to Theorem~\ref{T_Pw}\ref{T_Pw_2_new}, we can then define a linear operator $\widetilde L\colon\Dcal\to C_0(\Xcal)$ by the formula
\begin{equation}\label{L_tilde}
\widetilde L\widetilde f = L(\widetilde f\circ T)\circ T^{-1}.
\end{equation}

\begin{lemma} \label{L_Dprops}
We have the following properties.
\begin{enumerate}
\item\label{L_Dprops:3} $\Dcal$ is dense in $C_0(\Xcal)$,
\item\label{L_Dprops:4} for every $\widetilde f\in\Dcal$ and every $\nu^*\in\Xcal$, there exist measures $\nu_n\in T(\Pcal_w)$ with $\nu_n\Rightarrow\nu^*$ and $\widetilde f(\nu_n)=\widetilde f(\nu^*)$ for all $n\in\N$.
\end{enumerate}
\end{lemma}

\begin{proof}
\ref{L_Dprops:3}: This follows from the Stone--Weierstrass theorem once we show that $\Dcal$ separates points and vanishes nowhere on $\Xcal$. Any $\nu\in\Xcal$ satisfies $\langle w^{-1},\nu\rangle=1$, which implies that $\nu(E)>0$. It is then clear that some element of $\Dcal$ is nonzero at $\nu$. Thus $\Dcal$ vanishes nowhere. Next, take $\nu_1,\nu_2\in\Xcal$ such that $\langle \varphi,\nu_1\rangle e^{-\langle 1,\nu_1\rangle}=\langle \varphi,\nu_2\rangle e^{-\langle 1,\nu_2\rangle}$ for all $\varphi\in C^\infty_c(\R^d)$. By considering a sequence $\varphi_n\uparrow w^{-1}$ and using that $\langle w^{-1},\nu_i\rangle=1$, $i=1,2$, we deduce that $\langle 1,\nu_1\rangle=\langle 1,\nu_2\rangle$. Thus $\langle \varphi,\nu_1\rangle =\langle \varphi,\nu_2\rangle $ for all $\varphi\in C^\infty_c(\R^d)$, which implies that $\nu_1(\cdot\cap E)=\nu_2(\cdot\cap E)$. It follows that $\nu_1=\nu_2$, so that $\Dcal$ separates points as required.

\ref{L_Dprops:4}: If $\nu^*$ itself lies in $T(\Pcal_w)$, simply take $\nu_n=\nu^*$ for all $n$. We thus assume that this is not the case, which means that $\nu^*=\nu_0+\lambda\delta_\Delta$ for some $\nu_0\in M_+(E)$ and $\lambda>0$. Fix now a sequence $(x_n)_{n\in\N}\subseteq E$ with $|x_n|\to\infty$, or equivalently, $x_n\to\Delta$.
Since $w^{-1}(\Delta)=0$, we have $\langle w^{-1},\nu_0\rangle=\langle w^{-1},\nu^*\rangle=1$. 
For numbers $t_n\in(1,\infty)$ to be determined later, define the measures
\begin{equation} \label{eq_nu_n}
\nu_n = (1-t_n^{-1})\nu_0 + t_n^{-1} w(x_n) \delta_{x_n}.
\end{equation}
Since $\langle w^{-1},\nu_n\rangle=1$, these measures lie in $T(\Pcal_w)$. Fix any $\widetilde f\in\Dcal$ and observe that we have
\[
\widetilde f(\nu) = p(\langle \varphi_1,\nu\rangle e^{-\langle 1,\nu\rangle}, \ldots, \langle \varphi_m,\nu\rangle e^{-\langle 1,\nu\rangle})
\]
for some polynomial $p$ on $\R^m$ and some $\varphi_1,\ldots,\varphi_m\in C^\infty_c(\R^d)$. For all sufficiently large $n$, $x_n$ lies outside the supports of all the $\varphi_i$. For all such $n$, we have
\[
\langle \varphi_i,\nu_n\rangle e^{-\langle 1,\nu_n\rangle} = \langle \varphi_i,\nu_0\rangle (1-t_n^{-1})e^{-(1-t_n^{-1})\langle 1,\nu_0\rangle - t_n^{-1} w(x_n)}, \quad i=1,\ldots,m.
\]
On the other hand, we have
\[
\langle \varphi_i,\nu^*\rangle e^{-\langle 1,\nu^*\rangle} = \langle \varphi_i,\nu_0\rangle e^{-\langle 1,\nu_0\rangle - \lambda}, \quad i=1,\ldots,m.
\]
Therefore, if $t_n$ is chosen so that
\begin{equation} \label{eq_wx_ncondition}
w(x_n) = t_n \log(1-t_n^{-1}) + \langle 1,\nu_0\rangle + t_n \lambda,
\end{equation}
it follows that $\widetilde f(\nu_n)=\widetilde f(\nu^*)$. To see that this is possible, let $\alpha(t)$ denote the right-hand side of \eqref{eq_wx_ncondition}, with $t_n$ replaced by $t$. Then $t\mapsto\alpha(t)$ is continuous and strictly increasing on $(1,\infty)$ with $\lim_{t\to\infty}\alpha(t)=\infty$ and $\lim_{t\to-\infty}\alpha(t)=-\infty$. Therefore $\alpha$ has a continuous inverse $\alpha^{-1}(s)$ which satisfies $\lim_{s\to\infty}\alpha^{-1}(s)=\infty$. We now define
\[
t_n = \alpha^{-1}(w(x_n)).
\]
Since $\lim_{|x|\to\infty}w(x)=\infty$, we have $t_n\to\infty$, and since \eqref{eq_wx_ncondition} holds, we have $t_n^{-1}w(x_n)\to\lambda$. It is then clear from \eqref{eq_nu_n} that $\nu_n\Rightarrow \nu^*$. Therefore, after discarding the finitely many $\nu_n$ for which $x_n$ lies in the support of some $\varphi_i$, the measures $\nu_n$ satisfy the desired properties.
\end{proof}

\begin{lemma}
The operator $\widetilde L$ satisfies the positive maximum principle on $\Xcal$.
\end{lemma}

\begin{proof}
Let $\widetilde f\in\Dcal$ and $\nu^*\in \Xcal$ be such that $\widetilde f(\nu^*)=\max_\Xcal \widetilde f \ge 0$. By Lemma~\ref{L_Dprops}\ref{L_Dprops:4}, there exist measures $\nu_n\in T(\Pcal_w)$  with $\nu_n\Rightarrow\nu^*$ and $\widetilde f(\nu_n)=\widetilde f(\nu^*)$ for all $n\in\N$. In particular, we have $\widetilde f(\nu_n)=\max_{T(\Pcal_w)} \widetilde f \ge 0$ for all $n$. Thus, the function $f=\widetilde f\circ T\in\Dcal_w$ attains a nonnegative maximum over $\Pcal_w$ at the point $\mu_n=T^{-1}(\nu_n)$. Since $L$ satisfies the positive maximum principle on $\Pcal_w$, we get $\widetilde L\widetilde f(\nu_n) = L f(\mu_n) \le 0$. Sending $n$ to infinity and using that $\widetilde L\widetilde f$ is continuous on $\Xcal$ yields $\widetilde L\widetilde f(\nu^*)\le 0$. This shows that $\widetilde L$ satisfies the positive maximum principle on $\Xcal$, as claimed.
\end{proof}

\begin{proof}[Proof of Theorem~\ref{T_Pw}]
We have established that $\Xcal$ is locally compact, that $\Dcal\subseteq C_0(\Xcal)$ is dense, and that $\widetilde L(\Dcal)\subseteq C_0(\Xcal)$. Since $\widetilde L$ satisfies the positive maximum principle on $\Xcal$, Theorem~\ref{T:existence1} yields a possibly killed solution to the martingale problem for $(\widetilde L,\Dcal,\Xcal)$ for any initial condition $\nu\in\Xcal$. The state space $\Xcal^\dag$ for the possibly killed solution is the one-point compactification of $\Xcal$, and $\widetilde L^\dag$ and $\Dcal^\dag$ are as in Section~\ref{secmgprob}.

Fix $\mu\in \Pcal_w$ and let $Y$ be a solution with initial condition $\nu_0=T(\mu)$. We may suppose that $\dag$ is an absorbing state, that is, $Y_t=\dag$ for all $t\ge\inf\{t\ge0\colon\text{$Y_t=\dag$ or $Y_{t-}=\dag$}\}$. Assume for the moment that $Y$ actually takes values in $T(\Pcal_w)\cup\{\dag\}$. We can then define $X=T^{-1}(Y)$, with the convention $T^{-1}(\dag)=\dag$. For any $f\in\Dcal_w$, we have $\widetilde f=f\circ T^{-1}\in\Dcal$ as well as $Lf=\widetilde L\widetilde f\circ T$. These identities hold on $\Pcal_w\cup\{\dag\}$. We thus obtain that
\[
f(X_t)-f(X_0)-\int_0^t Lf(X_s)ds = \widetilde f(Y_t) - \widetilde f(Y_0) - \int_0^t \widetilde L\widetilde f(Y_s)ds, \quad t\ge0.
\]
Since the right-hand side is a local martingale, it follows that $X$ is a possibly killed solution to the martingale problem for $(L,\Dcal_w,\Pcal_w)$ with initial condition $\mu$.

We must still argue that $Y$ takes values in $T(\Pcal_w)\cup\{\dag\}$. Fix any $c\geq\max\{1,\langle 1,\nu_0\rangle\}$, and define the stopping time $\tau =\inf\{t\geq0:\langle 1,Y_t\rangle>c\}$, with the convention $\langle1,\dag\rangle=\infty$. Since $Y$ is a possibly killed solution to the martingale problem, an application of the optional stopping theorem yields
\[
\E[ \widetilde f(Y_{t\wedge\tau})] = \widetilde f(\nu_0) + \int_0^t \E[ \widetilde L^\dag \widetilde f(Y_s)1_{\{s<\tau\}}]ds
\]
for every $t\ge0$ and $\widetilde f\in\Dcal$. Since $Y_s\in\Xcal_c$ for $s<\tau$, we obtain
\begin{equation}\label{eq_Pw_proof1}
\E[ \widetilde f(Y_{t\wedge\tau})1_{\{Y_{t\wedge \tau}\neq \dag\}}] = \widetilde f(\nu_0) + \int_0^t \E[ \widetilde g(Y_s)1_{\{s<\tau\}}]ds
\end{equation}
for every $t\ge0$ and $(\widetilde f,\widetilde g)$ in the restricted graph \eqref{T_Pw_3_eq1}. Since $\widetilde f(\dag)=0$, the indicator on the left-hand side of \eqref{eq_Pw_proof1} is redundant. By dominated convergence, \eqref{eq_Pw_proof1} remains true for all $(\widetilde f,\widetilde g)$ in the bp-closure of the restricted graph \eqref{T_Pw_3_eq1}. Now the indicator is needed, since these functions are not defined at $\dag$.

Let now $\widetilde f$ and $(\widetilde f_m,\widetilde g_m)$ be as given in Theorem~\ref{T_Pw}\ref{T_Pw_3_new}. By dominated convergence, \eqref{eq_Pw_proof1}, and the conditions in Theorem~\ref{T_Pw}\ref{T_Pw_3_new}, we obtain
\begin{align*}
\E[\widetilde f(Y_{t\wedge \tau})1_{\{Y_{t\wedge \tau}\neq \dag\}}]
&=\lim_{m\to \infty}\E[\widetilde f_m(Y_{t\wedge \tau})1_{\{Y_{t\wedge \tau}\neq \dag\}}] \\
&=\lim_{m\to \infty}\left( \widetilde f_m(\nu_0) + \int_0^t \E[ \widetilde g_m(Y_s)1_{\{s<\tau\}}]ds \right) \\
&\le \widetilde f(\nu_0) + \int_0^t \E[ \limsup_{m\to \infty} \widetilde g_m(Y_{s})^+ 1_{\{s<\tau\}} ]ds \\
&\leq c' \int_0^t \E[\widetilde f(Y_s)1_{\{s<\tau\}}]ds\\
&\leq c' \int_0^t \E[\widetilde f(Y_{s\wedge \tau})1_{\{Y_{s\wedge \tau}\neq \dag\}}]ds.
\end{align*}
By the Gronwall inequality (see Theorem 5.1 in the Appendixes of \cite{EK:05}) we conclude that $\E[\widetilde f(Y_{t\wedge \tau})1_{\{Y_{t\wedge \tau}\neq \dag\}}]=0$. Together with the properties of $\widetilde f$ this implies that for $t<\tau$, $Y_t(\{\Delta\})=0$. Thus for $t<\tau$, $Y_t$ takes values in $T(\Pcal_w)$. Since the constant $c$ was arbitrarily large, and since $\dag$ is an absorbing state, it follows that $Y$ takes values in $T(\Pcal_w)\cup\{\dag\}$, as desired.

Finally, suppose condition \ref{T_Pw_6} in Theorem~\ref{T_Pw} is in force, and consider the functions $f_n$ given there. Let $X$ be a possibly killed solution to the martingale problem for $(L,\Dcal_w,\Pcal_w)$ with initial condition $\mu\in\Pcal_w$. We then get
\[
\E[1_{\{X_t\neq\dag\}}]=\lim_{n\to \infty}\E[f_n(X_t)]=\lim_{n\to \infty} \E\left[f_n(\mu)+\int_0^t  L^\dag f_n(X_s) ds\right]
\geq  1_{\{\mu\neq\dag\}}=1
\]
for every fixed $t$. As a result, $X_t\in\Pcal_w$ for all $t\ge0$, as claimed.
\end{proof}

\section{L\'evy type operators}\label{S_Levy_type}

Operators $L\colon\Dcal_w\to C(\Pcal_w)$ that satisfy the positive maximum principle are integro-differential operators of {\em L\'evy type}, which we now introduce. Such operators involve derivatives of functions $f$ of measure arguments, and we define
\[
\partial_x f(\mu) = \lim_{\varepsilon\to0}\frac{f(\mu + \varepsilon\delta_x)-f(\mu)}{\varepsilon}
\]
for all $(x,\mu)\in\R^d\times M(\R^d)$ for which the limit is well-defined. We write $\partial f(\mu)$ for the map $x\mapsto\partial_xf(\mu)$. Iterated derivatives are written $\partial^k_{x_1,\ldots,x_k}f(\mu)=\partial_{x_1}\cdots\partial_{x_k}f(\mu)$ whenever they exist, and we write $\partial^kf(\mu)$ for the corresponding map. Define also the function space
\[
C^\infty_w = \text{linear span of $w$ and $C^\infty_c(\R^d)$.}
\]
\begin{example}
If $f(\mu)=\langle\varphi,\mu\rangle e^{-\langle w,\mu\rangle}$ for some function $\varphi\colon\R^d\to\R$, then
\[
\partial_x f(\mu) = (\varphi(x) - \langle\varphi,\mu\rangle w(x)) e^{-\langle w,\mu\rangle}
\]
for any $x\in\R^d$ and any $\mu\in M(\R^d)$ such that $\varphi$ and $w$ are $\mu$-integrable. One also has the product rule $\partial(fg)=f\partial g + g\partial f$. In particular, every test function in $f\in\Dcal_w$ is infinitely many times differentiable and, for each $k$, the $k$-th derivative $\partial^k_{x_1,\ldots,x_k}f(\mu)$ is jointly continuous in $(x_1,\ldots,x_k,\mu)\in\R^k\times\Pcal_w$, and the map $\partial^kf(\mu)$ lies in $(C^\infty_w)^{\otimes k}$ for every fixed $\mu$.
\end{example}

We say that $L$ is of {\em L\'evy type} if it acts on test functions $f\in\Dcal_w$ by
\begin{equation}\label{eq_Levy_type}
\begin{aligned}
Lf(\mu) = - \kappa_\mu  f(\mu) &+ \langle B_\mu(\partial f(\mu)), \mu\rangle + \frac12 \langle Q_\mu(\partial^2f(\mu)), \mu^2\rangle\\
& \qquad+ \int_{\Pcal_w}(f(\nu)-f(\mu)-\langle\partial f(\mu),\chi(\nu-\mu)\rangle) N(\mu,d\nu),
\end{aligned}
\end{equation}
where, for each $\mu\in\Pcal_w$, the following conditions are imposed to ensure that the right-hand side is well-defined:
\begin{itemize}
\item $ \kappa_\mu \in\R_+$.
\item $B_\mu$ is a linear operator from $C^\infty_w$ to $L^1(E,\mu)$.
\item $Q_\mu$ is a linear operator from $C^\infty_w\otimes C^\infty_w$ to $L^1(E\times E,\mu\otimes\mu)$ with $\langle Q_\mu(\varphi\otimes \varphi),\mu^2\rangle\geq0$ for all $\varphi\in C^\infty_w$.
\item $N(\mu,d\nu)$ is a measure on $\Pcal_w$ with $\int_{\Pcal_w}1\land\langle\varphi,\mu-\nu\rangle^2 N(\mu,d\nu)<\infty$ for all $\varphi\in C^\infty_w$. In \eqref{eq_Levy_type}, $\chi(\nu-\mu)=(\nu-\mu)\rho(\langle w,\nu-\mu\rangle)$, with $\rho\colon\R\to[0,1]$ being a smooth function supported on $[-2,2]$ and equal to one on $[-1,1]$, acts as a truncation function for the large jumps.
\end{itemize}
If the martingale problem for $(L,\Dcal_w,\Pcal_w)$ has a possibly killed solution $X$ for any initial conditions $\mu$, then these objects govern the killing, drift, diffusion, and jump behavior of $X$, as one would expect.

\section{Verifying the positive maximum principle}

The positive maximum principle is convenient because it is often easy to verify in practice. The key tool for doing so when the operator is of L\'evy type are the following optimality conditions for functions in $\Dcal_w$.

\begin{theorem}\label{T_optimality}
Fix $f\in\Dcal_w$ and $\overline\mu\in\Pcal_w$ such that $f(\overline\mu)=\max_{\Pcal_w}f$.
\begin{enumerate}
\item\label{T_optimality:1} $\langle \partial f(\overline\mu),\mu\rangle = \sup_E\partial f(\overline\mu)$ for all $\mu\in\Pcal_w$ such that $\supp(\mu) \subseteq \supp(\overline\mu)$. In particular, $\partial_xf(\overline\mu) = \sup_E\partial f(\overline\mu)$ for all $x\in\supp(\overline\mu)$.
\item\label{T_optimality:2} $\langle \partial^2 f(\overline\mu),\mu^2\rangle \le0$ for all $\mu\in\Pcal_w-\Pcal_w$ such that $\langle1,\mu\rangle=0$ and $\supp(|\mu|) \subseteq \supp(\overline\mu)$. In particular,
\[
\partial^2_{xx}f(\overline\mu) + \partial^2_{yy}f(\overline\mu) - 2\partial^2_{xy}f(\overline\mu) = \langle\partial^2f(\overline\mu),(\delta_x-\delta_y)^2\rangle \le0, \quad x,y\in\supp(\overline\mu).
\]
\item\label{T_optimality:3b} Let $\tau\colon\R^d\to\R^{d\times d}$ be $C^1$ and satisfy the condition 
\begin{equation}\label{eqn21}
\text{$x\in E$, $\varphi\in C^\infty_c(\R^d)$, $\varphi(x)=\max_E \varphi$} \quad\Longrightarrow\quad \tau(x)^\top\nabla \varphi(x)=0.
\end{equation}
Define $A_\tau(\varphi)=\sum_{j=1}^d\tau_j^\top\nabla(\tau_j^\top\nabla\varphi)$ and $\aleph_\tau(\varphi\otimes\varphi)=\tr((\tau^\top\nabla\varphi)\otimes(\tau^\top\nabla\varphi)^\top)$ for each $\varphi\in C^\infty_w$, where $\tau_j$ denotes the $j$-th column of $\tau$. Assume the induced linear operators $A_\tau$ and $\aleph_\tau$ map $C^\infty_w$ and $C^\infty_w\otimes C^\infty_w$ to $L^1(\R^d,\overline\mu)$ and $L^1(\R^d\times\R^d,\overline\mu\otimes\overline\mu)$, respectively. Then
\begin{equation}\label{eqn16}
\langle A_\tau(\partial f(\overline\mu)),\overline\mu\rangle + \langle \aleph_\tau(\partial^2f(\overline\mu)),\overline\mu^2\rangle\le0.
\end{equation}
\end{enumerate}
If $E=\R^d$ one has the following improvement of \ref{T_optimality:3b}:
\begin{enumerate}[resume]
\item\label{T_optimality:4} Let $E=\R^d$. Then \ref{T_optimality:3b} remains true if $\tau$ is assumed to be continuous but not $C^1$. Note that in this case \eqref{eqn21} is vacuous and $A_\tau(\varphi)=\tr(\tau\tau^\top\nabla^2\varphi)$.
\end{enumerate}
\end{theorem}

\begin{proof}
 Before we begin, we need a technical result. Fix $f\in \Dcal_w$ and $\varphi_1,\ldots,\varphi_n\in C_c^\infty(\R^d)$ such that $f(\mu)=\Phi(\langle \varphi_1,\mu\rangle,\ldots,\langle \varphi_n,\mu\rangle,\langle w,\mu\rangle)$ for some $\Phi\in C^\infty(\R^{m+1})$. Applying the classical Taylor approximation theorem to $\Phi$ then yields
\begin{equation}\label{eqn13}
f(\mu+\nu_ t)=f(\mu)+\langle \partial f(\mu),\nu_ t\rangle+\frac {1} 2 \langle \partial^2f(\mu),\nu_ t^2\rangle+o( t^2)
\end{equation}
for all $\mu,\nu_t\in M(E)$ such that $w\in L^1(E,|\mu|)\cap L^1(E,|\nu_t|)$,  $\langle w,\nu_t\rangle=O(t)$, and 
$\langle \varphi_i,\nu_t\rangle=O(t)$ for $i=1,\ldots,n$.

\ref{T_optimality:1} and \ref{T_optimality:2}: The proof is a slight modification of the proof of  Theorem~3.1 in \cite{CLS:18}.
Pick any $x\in\supp(\overline \mu)$, $y\in E$, and let $A_n$ be the ball of radius $1/n$ centered at $x$, intersected with $\supp(\overline \mu)$. Note that $\partial f(\overline\mu)\in C_w^\infty$ and that setting $\mu_n:= \overline \mu(\fdot\cap A_n)/\overline \mu(A_n)$ we get that $\overline \mu+ t(\delta_y-\mu_n)\in \Pcal_w$ for all $ t\in (0,\overline \mu(A_n))$ and that $(\mu_n)_{n\in \N}$ converges to $\delta_x$  in $\Pcal_w$. Following the proof of Theorem 3.1 in \cite{CLS:18} yields the result.

\ref{T_optimality:3b}: Assume  that $\tau$ is compactly supported.
We follow the idea of the proof of Proposition~4.1 in \cite{ABI:19}. Fix $x\in \R^d$, $j\in\{1,\ldots, d\}$, and let $z_x:\R_+\to\R^d$ be the solution of the ODE
$$z_x'=\tau_j(z_x)\qquad\text{and}\qquad z_x(0)=x.$$
By Proposition 2.5 in \cite{DF:04} we know that $z_x(t)\in E$ for all $x\in E$ and $t\geq0$. 
Observe that for all $\varphi\in C_w^\infty$ and with $\psi=\varphi\otimes\varphi$ we have that
\begin{align*}
\lim_{t\to 0}\frac {\varphi(z_x(t))-\varphi(x)-(z_x'(0)^\top\nabla \varphi(x))t}{t^2}&=\frac 1 2 \tau_j(x)^\top\nabla(\tau_j^\top\nabla\varphi)(x),\\
\lim_{t\to0}\frac{\psi(x,y)-\psi(z_x(t), y)-\psi(x, z_y(t))+\psi(z_x(t), z_y(t))}{t^2}&=e_j^\top(\tau(x)^\top\nabla\varphi(x)\nabla \varphi(y)^\top\tau(y))e_j.
\end{align*}
Define $\mu_t:=F^{t}_*\overline \mu$, the pushforward of $\overline\mu$ under $F^t$, where $F^{t}(y):=z_y(t)$ for $y\in E$. Clearly $\mu_t\in M_1(E)$, and since $\tau_j$ is compactly supported, $\mu_t\in \Pcal_w$. Moreover, the dominated convergence theorem yields
$\lim_{t\to 0}\frac 1 {t} \langle \varphi,\mu_t-\overline\mu\rangle=\langle \tau_j^\top\nabla\varphi,\overline\mu\rangle$ for all $\varphi\in C_w^\infty$. We thus get that \eqref{eqn13} holds true for $\mu:=\overline \mu$ and $\nu_ t:=\mu_t-\overline \mu$.
Since $\overline \mu$ maximizes $f$ over $\Pcal_w$, we then obtain
\begin{align*}
0&\geq f(\mu_t)-f(\overline \mu)\\
&=\langle \partial f(\overline \mu),\mu_t-\overline \mu\rangle+\frac 1 2 \langle \partial^2 f(\overline \mu),(\mu_t-\overline \mu)^2\rangle+o(t^2)\\
&=\langle \partial_{F^t} f(\overline \mu)-\partial f(\overline \mu),\overline \mu\rangle
+\frac 1 2 \langle \partial^2 f(\overline \mu)-2\partial_{\fdot,{F^t}}^2 f(\overline \mu)+\partial_{{F^t},{F^t}}^2 f(\overline \mu),\overline \mu^2\rangle+o(t^2).
\end{align*}
Recall that by Theorem~\ref{T_optimality}\ref{T_optimality:1} we know that  $\partial_{ y} f(\overline \mu)=\sup_{ E}\partial f(\overline \mu)$
for all $ y\in\supp(\overline \mu)$ and hence ${\tau(y)}^\top \nabla (\partial p(\overline \mu))(y)=0$ by \eqref{eqn21}. As a result, dividing the above expression by $t^2$, letting $t$ go to 0, applying the dominated convergence theorem, and summing over $1\leq j\leq d$ yields \eqref{eqn16}.
A further application of dominated convergence allows to remove the assumption that $\tau$ is compactly supported.


\ref{T_optimality:4}:  The proof follows the proof of \ref{T_optimality:3b} using  $F^{t}(x):=x+t\tau_j(x) $ for all $x\in \R^d$. 
\end{proof}

The property \eqref{eqn21} in Condition~\ref{T_optimality:3b} can be understood as requiring that $\tau$ is a possible diffusion matrix for a (possibly killed) $E$-valued diffusion process. Under slightly more regularity on $\tau$, it is a known result in the stochastic invariance literature that this holds if and only if \eqref{eqn21} is satisfied; see for instance Theorem 4.1 in \cite{DF:04}. The extension to the current generality, where $\tau$ is merely $C^1$, can be proven using Theorem~\ref{T:existence1} and proceeding as in \cite{ABI:19}. We do not elaborate on the details, as we have no need for this result in the present paper.

Condition \ref{T_optimality:3b} is in fact an application of a more general condition, which we report here for the case $w\equiv 1$. See Theorem~3.4 and Remark~5.7(i) in \cite{CLS:18} for more details and a proof. An analogous result can be proved for $w$ as in \eqref{eq_wtoinfty}.
\begin{lemma}\label{T_optimality:3} Fix $w\equiv 1$, $f\in\Dcal_w$, and $\overline\mu\in\Pcal_w$ such that $f(\overline\mu)=\max_{\Pcal_w}f$.
Let $A$ be the generator of a strongly continuous group of positive isometries of $\R+C_0(E)$, and assume the domain of $A$ and the domain of $A^2$ both contain $C_w^\infty$.
 Then
\[
\langle A^2(\partial f(\overline\mu)),\overline\mu\rangle + \langle (A\otimes A)(\partial^2 f(\overline\mu)),\overline\mu^2\rangle \le 0.
\]
\end{lemma}

\section{Verifying the technical conditions}\label{S_verify_tech}

We now turn to the technical conditions \ref{T_Pw_2_new}--\ref{T_Pw_6} in Theorem~\ref{T_Pw}. At the end of the section, we follow up on Remark~\ref{rem2} and consider the case $w\equiv 1$. Recall the embedding $T\colon\Pcal_w\to M_+(E^\Delta)$ defined in \eqref{eq_T_embedding}.

We now start to work towards concrete ways of checking these assumptions.

\begin{lemma}\label{lem2}
Suppose $L$ is of L\'evy type \eqref{eq_Levy_type}. Assume that the functions
\[
\mu\mapsto  \kappa_\mu  e^{-\langle w,\mu\rangle}, \qquad \mu\mapsto \langle B_\mu(\varphi),\mu\rangle e^{-\langle w,\mu\rangle}, \qquad \mu\mapsto \langle Q_\mu(\varphi\otimes\varphi),\mu^2\rangle e^{-\langle w,\mu\rangle},
\]
\[
\mu\mapsto e^{-\langle w,\mu\rangle} \int_{\Pcal_w} \left(\langle\varphi,\nu\rangle e^{-\langle w,\nu-\mu\rangle} - \langle\varphi,\mu\rangle - \langle \varphi-w\langle\varphi,\mu\rangle,\chi(\nu-\mu)\rangle\right) N(\mu,d\nu),
\]
and
\[
\mu\mapsto e^{-k\langle w,\mu\rangle} \int_{\Pcal_w} \left( \langle\varphi,\nu\rangle e^{-\langle w,\nu-\mu\rangle} - \langle\varphi,\mu\rangle \right)^k N(\mu,d\nu), \qquad k\geq2,
\]
are of $C_0$ type for every $\varphi\in C^\infty_w$. Then so is $Lf$ for every $f\in\Dcal_w$, that is, condition \ref{T_Pw_2_new} in Theorem~\ref{T_Pw} is satisfied.
\end{lemma}

\begin{proof}
Note that one can check that each term of $Lf$ in \eqref{eq_Levy_type} is of $C_0$ type separately. Recall that $\mu\mapsto e^{-\langle w,\mu\rangle}$ is of $C_0$ type, as are all the elements of $\Dcal_w$.

{\it Killing:} The result for $\mu\mapsto\kappa_\mu f(\mu)$ follows by noting that $f(\mu)=\langle \varphi,\mu\rangle e^{-\langle w,\mu\rangle} g(\mu)$ for some $ g$ constant or in  $\Dcal_w$, for some  $\varphi\in C_c^\infty(\R^d)$, and for all $\mu\in \Pcal_w$.
 
{\it Drift:} Observe that for
\begin{equation}\label{eqn17}
\text{$f(\mu)=\langle\varphi,\mu\rangle e^{-\langle w,\mu\rangle}$ with $\varphi\in C^\infty_c(\R^d)$}
\end{equation}
we have
$\partial f(\mu) = (\varphi  - \langle \varphi,\mu\rangle w )e^{-\langle w,\mu\rangle}$. One then sees that the given conditions ensure that $\mu\mapsto\langle B_\mu(\partial f(\mu)), \mu\rangle$ is of $C_0$ type. By the product rule and linearity in $f$, this result extends to each $f\in \Dcal_w$.

{\it Diffusion:} For $f$ as in \eqref{eqn17} we have $\partial^2f(\mu) = (\langle\varphi,\mu\rangle w\otimes w - 2w\otimes \varphi )e^{-\langle w,\mu\rangle}$. The polarization identity $w\otimes \varphi = \frac14((w+\varphi)\otimes(w+\varphi) - (w-\varphi)\otimes(w-\varphi))$ and the given conditions yield that $\mu\mapsto\langle Q_\mu(\partial^2 f(\mu)), \mu\rangle$ and $\mu\mapsto\langle Q_\mu(\partial f(\mu)\otimes \partial f(\mu)), \mu\rangle$ are of $C_0$ type.  Applying the product rule twice we get $\partial^2(fg)=f\partial^2 g +2\partial g\otimes\partial f+ g\partial^2 f$ for each $f,g\in \Dcal_w$.
Combined with linearity in $f$ this yields that $\mu\mapsto\langle Q_\mu(\partial^2 f(\mu)), \mu\rangle$ is of $C_0$ type for each $f\in \Dcal_w$.

{\it Jumps:} Finally, consider $g(\mu):=p(f(\mu))$ for some polynomial $p:\R\to\R$ and some $f$ as in \eqref{eqn17}. Since $p$ is a polynomial and 
$\langle\partial g(\mu),\chi(\nu-\mu)\rangle=p'(f(\mu))\langle \partial f(\mu),\chi(\nu-\mu)\rangle$, an application of the classical Taylor approximation theorem to $p$ yields
$$
\begin{aligned}
&g(\nu)-g(\mu)-\langle\partial g(\mu),\chi(\nu-\mu)\rangle\\
&\qquad=p'(f(\mu))\Big(f(\nu)-f(\mu)-\langle \partial f(\mu),\chi(\nu-\mu)\rangle\Big)
+\sum_{\ell=2}^k\frac{p^{(\ell)}(f(\mu))}{\ell!}(f(\nu)-f(\mu))^\ell,
\end{aligned}
$$
where $k$ denotes the degree of $p$. Thus the given conditions ensure that the last term of $Lg$ in \eqref{eq_Levy_type} is of $C_0$ type. By polarization, this is also true for all $f\in \Dcal_w$.
\end{proof}

We focus now on condition \ref{T_Pw_3_new} in Theorem~\ref{T_Pw}, assuming that \ref{T_Pw_1}--\ref{T_Pw_2_new} are satisfied. This ensures that the operator $\widetilde L\colon\Dcal\to C_0(\Xcal)$ in \eqref{eq_T_Dw}--\eqref{L_tilde} is well-defined.

To verify these conditions it is useful to first extend $L$ to a larger class of functions than $\Dcal_w$, and then search for appropriate sequences $\{f_m\}_{m\in\N}$ in this larger class. The precise notion of extension is given in the following definition.

For any compact subset $\Kcal\subset\Xcal$ we define the following restricted graph of $\widetilde L$:
\[
\gph_\Kcal(\widetilde L) = \{(\widetilde f,\widetilde L\widetilde f|_\Kcal) \colon \widetilde f\in\Dcal\}\subset C_0(\Xcal)\times C(\Kcal).
\]
\begin{definition}
We say that $\widetilde L$ can be extended to a function $\widetilde f\colon\Xcal\to\R$ if there is another function $\widetilde g\colon\Xcal\to\R$ such that $(\widetilde f,\widetilde g|_\Kcal)$ lies in the bp-closure of $\gph_\Kcal(\widetilde L)$ for every compact subset $\Kcal\subset\Xcal$. We say that $L$ can be extended to a function $f\colon\Pcal_w\to\R$ if $\widetilde L$ can be extended to a function $\widetilde f\colon\Xcal\to\R$ that satisfies $f=\widetilde f\circ T$.
\end{definition}

If $f=\widetilde f \circ T$, and if $(\widetilde f,\widetilde g|_\Kcal)$ lies in the bp-closure of $\gph_\Kcal(\widetilde L)$ for every compact subset $\Kcal\subset\Xcal$, we write $Lf$ for $\widetilde g\circ T$. Sometimes, it happens that the expression given in \eqref{eq_Levy_type} is well defined for  $f$ and  coincides with $\widetilde g\circ T$.
 Those cases will be particularly important for our purposes.

\begin{lemma}\label{lem3}
Suppose $L$ is of L\'evy type \eqref{eq_Levy_type} and satisfies conditions \ref{T_Pw_1}--\ref{T_Pw_2_new} of Theorem~\ref{T_Pw}. Assume $L$ can be extended to all functions $f$ in the algebra generated by $\Dcal_w$ and $e^{-\langle w,\mu\rangle}$, and that $Lf$ is given by \eqref{eq_Levy_type}. Assume also there exist $[0,1]$-valued functions $\psi_m\in C^\infty_c(\R^d)$ with the following properties:
\begin{itemize}
\item $\psi_m\to1$ pointwise.
\item The functions $h_m\colon\Pcal_w\to\R$ given by $h_m(\mu)=\langle B_\mu((1-\psi_m)w),\mu\rangle^+$, which are of $C_0$ type, satisfy $\limsup_{m\to\infty}h_m\circ T^{-1}\leq c'\langle 1_{\{\Delta\}},\fdot\rangle$ in the bounded pointwise sense on every compact subset of $\Xcal$, for some constant $c'$.
\end{itemize}
Then condition \ref{T_Pw_3_new} in Theorem~\ref{T_Pw} is satisfied.
\end{lemma}

\begin{proof}
We claim that 
$L$ can be extended to all maps $f:\Pcal_w\to\R$ of the form
\begin{equation}\label{eqn18}
f(\mu):= p\big(\langle \varphi_1,\mu\rangle e^{-\langle w,\mu\rangle},\ldots,\langle \varphi_n,\mu\rangle e^{-\langle w,\mu\rangle}\big)
\end{equation}
for some $p\in C^2(\R^n)$ with $p(0)=0$ and $\varphi_1,\ldots,\varphi_n\in \R+C^\infty_c(\R^d)$, and that $Lf$ is given by \eqref{eq_Levy_type}.
To see this, define functions
\[
f_m(\mu):=p_m\big(\langle \varphi_1,\mu\rangle e^{-\langle w,\mu\rangle},\ldots,\langle \varphi_n,\mu\rangle e^{-\langle w,\mu\rangle}\big)
\]
for some polynomials $p_m$ on $\R^n$ with $p_m(0)=0$ such that
$p_m\to p$, $\nabla p_m\to \nabla p$, and $\nabla^2p_m\to\nabla^2p$ uniformly on $[-R,R]^n$, where $R=\max_{i=1,\ldots,n}\sup_{\mu\in\Pcal_w}|\langle \varphi_i,\mu\rangle| e^{-\langle w,\mu\rangle}$. Then $f_m\to f$ and $Lf_m \to Lf$ uniformly on $\Pcal_w$, where $Lf$ is defined by \eqref{eq_Levy_type}. Therefore the functions $\widetilde f_m=f_m\circ T^{-1}$ and $\widetilde L\widetilde f_m = (Lf_m)\circ T^{-1}$ converge in the bounded pointwise sense (even uniformly) to $\widetilde f=f\circ T^{-1}$ and $Lf\circ T^{-1}$. This shows that $L$ can be extended to $f$, and that $Lf$ is given by \eqref{eq_Levy_type}.

Fix any $c\ge1$. Let $p_c\in C^2(\R_+)$ be such that
$1_{\{x\leq c\}}\leq p_c(x)\leq1_{\{x\leq2c\}}$, and 
set
$$f_m(\mu):=p_c(\langle w,\mu\rangle) \langle (1-\psi_m)w,\mu\rangle.$$
The function $f_m$ is of the form \eqref{eqn18}. Indeed, we can write $f_m$ as
$$
f_m(\mu)= p_1\big(\langle1,\mu\rangle e^{-\langle w,\mu\rangle}\big)- p_2\big(\langle 1,\mu\rangle e^{-\langle w,\mu\rangle}\big) \langle \psi_mw,\mu\rangle e^{-\langle w,\mu\rangle},
$$
where $p_1(x)=p_c(-\log(x))(-\log(x))$ and $p_2(x)=p_c(-\log(x))/x$ for $x>0$, and $p_1(x)=p_2(x)=0$ for $x\le0$.

We now define $\widetilde f_m=f_m\circ T^{-1}$, $\widetilde g_m=Lf_m\circ T^{-1} |_{\Xcal_c}$, and $\widetilde f(\nu)=p_c(\langle1,\nu\rangle)\nu(\{\Delta\})$, and prove that these functions satisfy the properties in Theorem~\ref{T_Pw}\ref{T_Pw_3_new}. Since $L$ can be extended to $f_m$, the pair $(\widetilde f_m,\widetilde g_m)$ lies in the bp-closure of $\gph_{\Xcal_c}(\widetilde L)$. Moreover, the $\widetilde f_m$ are uniformly bounded in $m$ because $0\le \widetilde f_m(\nu)=p_c(\langle 1,\nu\rangle) \langle (1-\psi_m),\nu\rangle\le 2c$. Next, for all $\mu\in \Pcal_w$ such that $\langle w,\mu\rangle\le c$, we have $f_m(\mu)= \langle (1-\psi_m)w,\mu\rangle$, $\partial f_m(\mu)=(1-\psi_m)w$, and $\partial^2f_m(\mu)=0$. Since also $\kappa_\mu$ is nonnegative we have
\begin{align*}
Lf_m(\mu) &\le \langle B_\mu((1-\psi_m)w),\mu\rangle + \int_{\Pcal_w} \big(p_c(\langle w,\nu\rangle)-1\big) \langle (1-\psi_m)w,\nu\rangle
N(\mu,d\nu) \\
&\le \langle B_\mu((1-\psi_m)w),\mu\rangle.
\end{align*}
Therefore $\widetilde g_m^+ \le h_m\circ T^{-1}|_{\Xcal_c}$, and our hypotheses imply that $\widetilde g_m^+$ is uniformly bounded in $m$, and that $\limsup_{m\to\infty} \widetilde g_m^+ \le c' \widetilde f|_{\Xcal_c}$ pointwise.  The dominated convergence theorem implies that $\widetilde f_m\to\widetilde f$ pointwise. Finally, it follows by inspection that $\widetilde f\ge0$, and that $\widetilde f(\nu)=0$ for $\nu\in\Xcal_c$ if and only if $\nu(\{\Delta\})=0$.
\end{proof}

The last condition to analyze is condition \ref{T_Pw_6}. 

\begin{lemma}\label{lem4}
Suppose $L$ is of L\'evy type \eqref{eq_Levy_type} and satisfies the conditions \ref{T_Pw_1}--\ref{T_Pw_2_new} of Theorem~\ref{T_Pw}.  Assume that for every function $f$ in the algebra generated by $\Dcal_w$ and $e^{-\langle w,\mu\rangle}$, $(f,Lf)$ with $Lf$ is given by \eqref{eq_Levy_type} lies in the bp-closure of the graph $\{(h,Lh)\colon h\in\Dcal_w\}$.
Assume also that $\kappa_\mu=0$ for all $\mu\in \Pcal_w$, and that we have the linear growth condition
\[
\langle w\otimes B_\mu(w),\mu^2\rangle^+ + \langle Q_\mu(w\otimes w),\mu^2\rangle + \int_{\Pcal_w}\langle w,\nu-\mu\rangle^2\land \langle w,\mu\rangle^2 \Nu(\mu,d\nu) \le c\langle w,\mu\rangle^2
\]
for all $\mu\in\Pcal_w$ and some constant $c$. Then condition \ref{T_Pw_6} in Theorem~\ref{T_Pw} is satisfied.
\end{lemma}

\begin{proof}
We showed in the proof of Lemma~\ref{lem3} that $L$ can be extended to all maps $f\colon\Pcal_w\to\R$ of the form \eqref{eqn18}, and that $Lf$ is given by \eqref{eq_Levy_type}. In fact, under our current assumptions, the argument shows that $(f,Lf)$ lies in the bp-closure (even the uniform closure) of the graph $\{(h,Lh)\colon h\in\Dcal_w\}$. In particular, these facts apply to the maps $f_n(\mu)=q(\langle w,\mu\rangle/n)$, where $q\in C_c^\infty(\R_+)$ is nonincreasing and satisfies $1_{[0,1]}\leq q\leq1_{[0,2]}$. It is clear that $f_n\to 1$ in the bounded pointwise sense. We must argue that $(Lf_n)^-\to0$ in the same sense. A direct computation gives
\begin{align*}
Lf_n(\mu)
& =q'(\langle w,\mu\rangle/n)\frac 1 n\langle  B_\mu(w),\mu\rangle +\frac 1 2 q''(\langle w,\mu\rangle/n)\frac 1 {n^2}\langle Q_\mu(w\otimes w),\mu^2\rangle\\
&\quad+\int_{\Pcal_w} \Big( q(\langle w,\nu\rangle/n)-q(\langle w,\mu\rangle/n)-q'(\langle w,\mu\rangle/n)\frac1n \langle w,\chi(\nu-\mu)\rangle \Big) \Nu(\mu,d\nu).
\end{align*}
Using the properties of $q$ and, in the last step, the assumed linear growth condition, we get for some constant $c'$ the lower bound,
\begin{align*}
& -c' \Big(\frac 1 n \langle B_\mu(w),\mu\rangle^+
+ \frac 1 {n^2}\langle Q_\mu(w\otimes w),\mu^2\rangle
+\int_{\Pcal_w}\frac{\langle w,\nu-\mu\rangle^2}{n^2}\land 1\,\Nu(\mu,d\nu)\Big)1_{\{\langle w,\mu\rangle\leq2n\}} \\
&\quad\ge -\frac{4c'}{\langle w,\mu\rangle^2} \Big(\langle w\otimes B_\mu(w),\mu^2\rangle^+
+ \langle Q_\mu(w\otimes w),\mu^2\rangle
+\int_{\Pcal_w}\langle w,\nu-\mu\rangle^2\land \langle w,\mu\rangle^2 \Nu(\mu,d\nu)\Big) \\
&\quad\ge -4 c' c.
\end{align*}
We deduce that $(Lf_n)^-$ is uniformly bounded in $n$, and it is clear from the first line that $(Lf_n)^-\to0$ pointwise.
\end{proof}

\begin{remark}
We now comment on the case $w\equiv1$, and consider the setting of Remark~\ref{rem2}. Lemma~\ref{lem2} would then be replaced by the requirement that $Lf$ can be extended to a function in $C(M_1(E^\Delta))$ for all $f\in \Dcal_w$. Concrete conditions when $L$ is of L\'evy type are that the maps $\mu\mapsto \kappa_\mu$, $\mu\mapsto B_\mu(\varphi)$, and $\mu\mapsto Q_\mu(\varphi\otimes\varphi)$ are continuous from $\Pcal_w$ to $\R$, $\R+C_0(\R^d)$, and $\R+C_0(\R^d)\otimes C_0(\R^d)$, respectively, and that $\mu\mapsto\int \langle \varphi,\nu-\mu\rangle^\ell\Nu(\mu,d\nu)$ can be extended to a function in $C(M_1(E^\Delta))$ for every $\ell\geq2$ and $\varphi\in C_w^\infty$. Next, Lemma~\ref{lem3} holds without the assumption that $L$ can be extended to all functions in the algebra generated by $\Dcal_w$ and $e^{-\langle w,\mu\rangle}$. In the proof, one simply takes $f_m(\mu)=\langle (1-\psi_m),\mu\rangle$.

In the case of condition \ref{T_Pw_6}, the result for $w\equiv 1$ is quite different from Lemma~\ref{lem4}. The reason is that, in contrast to the cases where condition \eqref{eq_wtoinfty} is in force,  for $w\equiv1$ the cemetery state $\dag$ is an isolated point. Thus a solution to the martingale problem can reach $\dag$ only by means of a jump. This has the consequence that \ref{T_Pw_6} can essentially be derived from \ref{T_Pw_3_new}. We now report a precise formulation of this statement.

\begin{lemma}
Suppose that $w\equiv1$, $L$ is of L\'evy type \eqref{eq_Levy_type} and satisfies conditions \ref{T_Pw_1}--\ref{T_Pw_3_new} of Theorem~\ref{T_Pw}. Then, if $\kappa_\mu=0$ for all $\mu\in \Pcal_w$, condition \ref{T_Pw_6} of Theorem~\ref{T_Pw} is satisfied.
\end{lemma}
\begin{proof}
Choose $(\widetilde f_m,\widetilde g_m)_{m\in \N}$ as in \ref{T_Pw_3_new} for $c=1$. Define $f_m=1-\widetilde f_m\circ T$ and $g_m=-\widetilde g_m\circ T$. Since $\kappa_\mu=0$, $(f_m,g_m)$ lies in the bp-closure of the graph of $L$. Moreover, $(f_m,g_m^-)\to(1,0)$ in the bounded pointwise sense.
\end{proof}
\end{remark}

\section{Applications of the main result}\label{three}

Take $E=\R^d$ and assume that $w(x)=|x|^p$ for $|x|>2$, where $p\in(0,\infty)$. Consider a linear operator $L\colon\Dcal_w\to C(\Pcal_w)$ of L\'evy type \eqref{eq_Levy_type} with $\kappa=0$, $\Nu=0$, and $B$ and $Q$ given by
\begin{align*}
B_\mu(\varphi) &= b_\mu^\top\nabla\varphi+\frac 12\tr((\sigma_\mu^2+\tau_\mu^2)\nabla^2\varphi),\\
Q_\mu(\varphi\otimes \varphi) &= (\tau_\mu\nabla \varphi)\otimes (\tau_\mu\nabla \varphi),
\end{align*}
for some maps $b\colon\Pcal_w\times \R^d\to \R^d$ and $\sigma,\tau\colon\Pcal_w\times \R^d\to \S^{d}$. Here we use the notation $u\otimes v=u_1\otimes v_1+\cdots+u_d\otimes v_d$ whenever $u,v\colon\R^d\to\R^d$.

\begin{theorem}\label{thm1}
Assume that $b_\mu(x)=\widetilde b_{w\mu}(x)$, $\sigma_\mu(x)=\widetilde \sigma_{w\mu}(x)$, $\tau_\mu(x)=\widetilde \tau_{w\mu}(x)$ for some continuous maps $\widetilde b\colon\Xcal\times\R^d\to \R^d$ and $\widetilde\sigma,\widetilde\tau\colon\Xcal\times\R^d\to \S^{d}$ such that 
\begin{equation}\label{eqn32}
\frac{\widetilde b_{\nu,i}(x)}{1+|x|},\quad \frac{\widetilde \sigma_{\nu,ij}(x)^2}{1+|x|^2},\quad\text{and}\quad  \frac{\widetilde \tau_{\nu,ij}(x)}{1+|x|},\qquad i,j\in\{1,\ldots, d\},
\end{equation}
are continuous as functions from $\Xcal$ to $\R+C_0(\R^d)$. Assume also that
\begin{equation}\label{eqn33}
\sup_{x\in\R^d} \frac{|x|\,|\widetilde b_\nu(x)| +| \widetilde \sigma_\nu(x)|^2 + |\widetilde \tau_\nu(x)|^2}{1+|x|^2}
\le c \langle 1,\nu\rangle^\gamma, \quad \nu\in\Xcal,
\end{equation}
for some constants $c,\gamma\ge0$. Then conditions \ref{T_Pw_1}--\ref{T_Pw_3_new} of Theorem~\ref{T_Pw} are satisfied, and thus there exists a possibly killed solution $X$ to the martingale problem for $(L,\Dcal_w,\Pcal_w)$ for every initial condition $\mu\in\Pcal_w$.
If one can take $\gamma=0$ in \eqref{eqn33}, then condition \ref{T_Pw_6} of Theorem~\ref{T_Pw} holds, and thus $X_t\in \Pcal_w$ for all $t\geq0$.
\end{theorem}

\begin{proof}
For notational simplicity we only prove the case $d=1$.
For all $\varphi\in C_c^\infty(\R)$ set 
$
f_\varphi(\mu)=\langle\varphi,\mu\rangle e^{-\langle w,\mu\rangle}
$
and recall that $\partial f_\varphi(\mu)= (\varphi  - \langle \varphi,\mu\rangle w )e^{-\langle w,\mu\rangle}$ and $\partial^2f_\varphi(\mu)= (\langle\varphi,\mu\rangle w\otimes w - 2w\otimes \varphi )e^{-\langle w,\mu\rangle}$. Define also $\widetilde f_\varphi=f_\varphi\circ T^{-1}$, and set 
$\widetilde B_{ \nu}(\varphi)= \widetilde b_{ \nu}\varphi'+\frac 12(\widetilde \sigma_{ \nu}^2+\widetilde \tau_{ \nu}^2)\varphi''$
and $\widetilde \Sigma_{ \nu}(\varphi)= \widetilde \tau_{ \nu} \varphi'$ for all $\nu\in \Xcal$ and $\varphi\in C_w^\infty$. 

Next, observe that $\varphi,\varphi'$, and $\varphi''$ are continuous and  $|\varphi'(x)x|/w(x)$ and $|\varphi''(x)x^2|/w(x)$ are bounded for all $\varphi\in C_w^\infty$. Condition \eqref{eqn32} then yields that for all such $\varphi$,
\begin{equation}\label{eqn32a}
\frac{\widetilde b_\nu(x)\varphi'(x)}{w(x)},\quad \frac{(\widetilde \sigma_\nu(x)^2+\tau_\nu(x)^2)\varphi''(x)}{w(x)},
\quad\text{and}\quad \frac{\widetilde \tau_\nu(x)\varphi'(x)}{w(x)}
\end{equation}
are continuous as functions from $\Xcal$ to $\R+C_0(\R)$, and condition \eqref{eqn33} yields that
\begin{equation}\label{eqn33a}
\sup_{x\in\R} \frac{ |\widetilde b_\nu(x)\varphi'(x)| + (\widetilde \sigma_\nu(x)^2 + \widetilde \tau_\nu(x)^2)|\varphi''(x)|+|\widetilde \tau_\nu(x)\varphi'(x)|}{w(x)}
\le c_\varphi \langle 1,\nu\rangle^\gamma, \quad \nu\in\Xcal,
\end{equation}
for some constant $c_\varphi$ depending on $\varphi$.
We are now ready to verify the conditions of Theorem~\ref{T_Pw}.

\ref{T_Pw_1}: The positive maximum principle follows directly form Theorem~\ref{T_optimality}\ref{T_optimality:1} and \ref{T_optimality:4}.

\ref{T_Pw_2_new}:
We verify the conditions of Lemma~\ref{lem2}. That is, in the current notation, we must check that the functions
\begin{equation}\label{eqn5}
\nu\mapsto \langle \widetilde B_\nu(\varphi)w^{-1},\nu\rangle e^{-\langle 1,\nu\rangle} \qquad \text{and}\qquad\nu\mapsto \langle \widetilde \Sigma_\nu(\varphi) w^{-1},\nu\rangle^2 e^{-\langle 1,\nu\rangle}
\end{equation}
are $C_0$ functions on $\Xcal$ for every $\varphi\in C^\infty_w$. To see that the function involving $\widetilde B_\nu$ is continuous, we write
\begin{align*}
\Big|\langle \widetilde B_\nu(\varphi) w^{-1},\nu\rangle-\langle \widetilde B_{\nu_n}(\varphi)w^{-1},\nu_n\rangle\Big|
&\leq \Big|\langle \widetilde B_\nu(\varphi)w^{-1},\nu-\nu_n\rangle\Big| \\
&\quad +\sup_{x\in\R}\Big| (\widetilde B_\nu(\varphi)(x)- \widetilde B_{\nu_n}(\varphi)(x))w^{-1}(x)\Big| \langle 1,\nu_n\rangle.
\end{align*}
If $\nu_n\Rightarrow\nu$, then the first term tends to zero since $\widetilde B_{\nu}(\varphi)w^{-1}\in \R+C_0(\R)$. The second term tends to zero due to \eqref{eqn32a}. A similar calculation shows that the function in \eqref{eqn5} involving $\widetilde\Sigma_\nu$ is continuous as well.

It remains to show that the functions in \eqref{eqn5} vanish at infinity. To see this, note that
\begin{equation}\label{eqn35}
\begin{aligned}
|\langle \widetilde B_\nu(\varphi)w^{-1},\nu\rangle|
&\leq \sup_{x\in\R}  \frac{|\widetilde b_\nu(x) \varphi'(x)| + \frac12 (\widetilde \sigma_\nu(x)^2+\widetilde \tau_\nu(x)^2) |\varphi''(x)|}{w(x)} \langle 1,\nu\rangle,\\
|\langle \widetilde Q_\nu(\varphi\otimes \varphi)w^{-2},\nu^2\rangle|
&\leq  \sup_{x\in\R} \frac{|\widetilde\tau_\nu(x) \varphi'(x)|^2}{w(x)^2} \langle 1,\nu\rangle^2
\end{aligned}
\end{equation}
for all $\varphi\in C_w^\infty$. Since the suprema in \eqref{eqn35} grow at most polynomially in $\langle1,\nu\rangle$ due to \eqref{eqn33a}, the functions in \eqref{eqn5} vanish at infinity.

\ref{T_Pw_3_new}:
We verify the conditions of Lemma~\ref{lem3}. We have already shown that $L$ satisfies conditions \ref{T_Pw_1}--\ref{T_Pw_2_new} of Theorem~\ref{T_Pw}. Let us show that $L$ can be extended to all functions $f$ in the algebra generated by $\Dcal_w$ and $e^{-\langle w,\mu\rangle}$, and that $Lf$ is given by \eqref{eq_Levy_type}.

Fix  $\psi_n(x):=\psi(x/n)$ for some $\psi\in C_c^\infty(\R)$ such that $1_{[-1,1]}\leq\psi\leq1_{[-2,2]}$ and set
$\widetilde f_n:=p(\widetilde f_{\varphi_1},\ldots,\widetilde f_{\varphi_k},\widetilde f_{\psi_n})$ and $\widetilde f:=p(\widetilde f_{\varphi_1},\ldots,\widetilde f_{\varphi_k},e^{-\langle 1,\fdot\rangle})$
for an arbitrary polynomial $p:\R^{k+1}\to\R$ with $p(0)=0$ and some $\varphi_1,\ldots,\varphi_k\in C_c^\infty(\R)$.
Note that $\widetilde f_n$ converges to 
$\widetilde f$ on $\Xcal $ bounded pointwise. 
Next, observe that condition \eqref{eqn33} yields 
\begin{equation}\label{eqn301}
\begin{aligned}
|\langle\widetilde B_\nu(\psi_n)w^{-1},\nu\rangle|
&\leq 2
\sup_{x\in\R}\Big|\frac{\widetilde b_\nu(x)}{n}1_{[-2n,2n]}(x)+\frac{\widetilde\sigma_\nu(x)^2+\widetilde\tau_\nu(x)^2}{2n^2}1_{[-2n,2n]}(x)\Big|\langle w^{-1},\nu\rangle\\
&\leq c''
\sup_{x\in\R}\frac{|x|\,|\widetilde b_\nu(x)|+ \widetilde\sigma_\nu(x)^2+\widetilde\tau_\nu(x)^2}{1+x^2} \\
&\leq c''c \langle 1,\nu\rangle^\gamma.
\end{aligned}
\end{equation}
for some constant $c''$.
Similarly, one can bound
$$
|\langle \widetilde Q_\nu(\psi_n\otimes \psi_n)/w^2,\nu^2\rangle|\quad\text{ and }\quad
|\langle \widetilde Q_\nu(\varphi\otimes \psi_n)/w^2,\nu^2\rangle|
$$
by $c'''c\langle1,\nu\rangle^\gamma$, for some constant $c'''$ and  all $\varphi\in\{\varphi_1,\ldots,\varphi_k,w\}$. Since both bounds grow at most polynomially in $\langle 1,\nu\rangle$, the product rule and the polarization identity yield that $(\widetilde L\widetilde f_n)_{n\in \N}$ is a bounded sequence in $C_0(\Xcal)$. Moreover, the dominated convergence theorem shows that $\widetilde L\widetilde f_n(\nu)\to \widetilde g(\nu)$ for all $\nu\in\Xcal$, where $\widetilde g\circ T(\mu)=Lf(\mu)$ as given by \eqref{eq_Levy_type} for $f=\widetilde f\circ T=p( f_{\varphi_1},\ldots, f_{\varphi_k}, e^{-\langle w,\fdot\rangle})$. This proves that  $(\widetilde f,\widetilde g)$ lies in the bp-closure of the graph $\{(\widetilde f,\widetilde L\widetilde f) \colon \widetilde f\in\Dcal\}$, and thus that $\widetilde L$ can be extended to $\widetilde f$.
 By definition, this implies that $L$ can be extended to $f$ with $Lf$ given by \eqref{eq_Levy_type}.

Let now  $h_m$ be as in Lemma~\ref{lem3} and note that
$$h_m\circ T^{-1}(\nu)=\langle\widetilde B_\nu((1-\psi_m)w) w^{-1},\nu\rangle^+.$$
The first inequality in \eqref{eqn35}, condition \eqref{eqn33a}, and the reasoning in \eqref{eqn301} imply that $(h_m|_\Kcal)_{m\in \N}$  is a bounded sequence in $C(\Kcal)$ for every compact set $\Kcal$. Moreover,
$$\lim_{m\to\infty}\frac{\widetilde B_\nu((1-\psi_m)w)}w=\Big(\widetilde b_\nu(\Delta) \frac{w'(\Delta)}{w(\Delta)}+\frac 12(\widetilde\sigma_\nu^2(\Delta)+\widetilde\tau_\nu^2(\Delta))\frac{w''(\Delta)}{w(\Delta)}\Big) 1_{\{\Delta\}},
$$
which is well defined by \eqref{eqn32a}. Write the right-hand side as $c' 1_{\{\Delta\}}$ for a constant $c'$.  By the dominate convergence theorem we can conclude that  $h_m\circ T^{-1}\to c'\langle 1_{\{\Delta\}},\fdot\rangle$ in the bounded pointwise sense on every compact subset of $\Xcal$.  Thus the conditions of Lemma~\ref{lem3} hold, as required.

\ref{T_Pw_6}:
We verify the conditions of Lemma~\ref{lem4} under the additional assumption that one can take $\gamma=0$ in \eqref{eqn33}. We already know that $L$ satisfies the conditions \ref{T_Pw_1}--\ref{T_Pw_2_new} of Theorem~\ref{T_Pw} and that $\kappa_\mu=0$ for all $\mu\in \Pcal_w$. 
We show now that for every function $f$ in the algebra generated by $\Dcal_w$ and $e^{-\langle w,\mu\rangle}$, the pair $(f,Lf)$ with $Lf$ given by \eqref{eq_Levy_type} lies in the bp-closure of the graph $\{(h,Lh)\colon h\in\Dcal_w\}$.
To do this, we just need to follow the first part of the proof of \ref{T_Pw_3_new}. Indeed, setting $ f_n= \widetilde f_n\circ T$ we get that $(f_n,Lf_n)\in\{(h,Lh)\colon h\in\Dcal_w\}$ converges bounded pointwise to $(f,Lf)$ for $Lf$ as given in  \eqref{eq_Levy_type}. 

The last condition of Lemma~\ref{lem4} to be verified is the linear growth. By \eqref{eqn33a} with $\varphi=w$ and $\gamma=0$ we compute
\begin{align*}
\langle w\otimes B_\mu(w),\mu^2\rangle
&=\langle w,\mu\rangle \langle b_\mu w'+\frac 12(\sigma_\mu^2+\tau_\mu^2)w'',\mu\rangle
\leq c_w \langle w,\mu\rangle^2\\
 \langle Q_\mu(w\otimes w),\mu^2\rangle
&=\langle \tau_\mu w',\mu\rangle^2
\leq c_w \langle w,\mu\rangle^2.
\end{align*}
The claim follows.
\end{proof}

\begin{remark}
As will be explored further in Section~\ref{IIIsec42}, the linear operator $L$ introduced at the beginning of the section coincides with the generator of the conditional distribution $X_t=\P(Z_t\in\fdot\mid\Fcal_t^0)$ of a solution of a McKean--Vlasov equation with common noise,
$$d Z_t=b_{X_t}(Z_t)dt+\sigma_{X_t}(Z_t)dW_t+\tau_{X_t}(Z_t)dW_t^0,\quad X_t=\P(Z_t\in\fdot\mid\Fcal_t^0),$$
where $ \Fcal_t^0:=\sigma(W^0_s,s\leq t)$.
The same result provided by Theorem~\ref{thm1} can be obtained when the common noise is replaced by a common jump mechanism. For example, consider a poisson random measure $\Pcal^0(dt,dy)$ with compensator $F(dy)dt$ for some probability measure $F$ supported on $\R$, and let $(X,Z)$ satisfy the McKean--Vlasov equation
$$d Z_t=b_{X_t}(Z_t)dt+\sigma_{X_t}(Z_t)dW_t+\int \ell_{X_{t-}}(Z_{t-},y)\Pcal^0(dt,dy),\quad X_t=\P(Z_t\in\fdot\mid\Fcal_t^0),$$
where $ \Fcal_t^0:=\sigma(\Pcal^0([0,s],dy),s\leq t)$. Here $\ell_\mu(x,y)$ describes the sizes of the common jumps, which we assume are confined to a cube $[0,c]^d$ for some $c>0$. The generator of the probability measure valued process $X$ is then the linear operator of L\'evy type \eqref{eq_Levy_type} given by
$$
Lf(\mu) =\langle B_\mu(\partial f(\mu)), \mu\rangle + \int_{\R}f(\gamma(\mu,y))-f(\mu) F(dy)
$$
for $
B_\mu(\varphi) = b_\mu^\top\nabla\varphi+
\frac 12\tr(\sigma_\mu^2\nabla^2\varphi)+
\int (\varphi(\fdot+\ell_{\mu}(\fdot,y))-\varphi) F(dy),
$
where
$$\gamma(\mu,y):=(\fdot+\ell_\mu(\fdot,y))_*\mu\in \Pcal_w.$$
Note that since $F$ is a probability measure, we are free to choose $\chi\equiv 0$ as truncation function.
Suppose now that  $b$ and $\sigma$ satisfy the conditions of Theorem~\ref{thm1} with $\gamma=0$. Assume also that $\ell_\mu(x,y)=\widetilde\ell_{w\mu}(x,y)$ for some continuous map $(\nu,x)\mapsto\widetilde\ell_\nu(x,y)$ from $\Xcal\times \R^d$ to $[0,c]^d\setminus\{0\}$ such that
$\widetilde\ell_{\nu}(x,y)$
is a continuous map from $\Xcal$ to the space $C(\R^d,[0,c]^d)$ of continuous functions from $\R^d$ to $[0,c]^d$.
Then there exists a solution $X$ to the martingale problem for $(L,\Dcal_w,\Pcal_w)$ for every initial condition $\mu\in\Pcal_w$. This result can be proved following the proof of Theorem~\ref{thm1}.
\end{remark}

The next corollary follows directly from Theorem~\ref{thm1}.
\begin{corollary}
Suppose that $b$, $\sigma$, and $\tau$ do not depend on $\mu$ and
$$\frac{ b_{i}(x)}{1+|x|},\quad \frac{ \sigma_{ij}(x)^2}{1+|x|^2},\quad  \frac{ \tau_{ij}(x)}{1+|x|}\in  \R+C_0(\R^d),\qquad i,j\in\{1,\ldots, d\}.$$
Then there exists a  solution $X$ to the martingale problem for $(L,\Dcal_w,\Pcal_w)$ for every initial condition $\mu\in\Pcal_w$.
\end{corollary}

We consider now a different linear operator $L\colon\Dcal_w\to C(\Pcal_w)$ of L\'evy type \eqref{eq_Levy_type} with $\kappa=0$, $\Nu=0$, and $B$ and $Q$ given by
\begin{equation}\label{eqn3}
B_\mu(\varphi) = b_\mu^\top\nabla\varphi+\frac 12\tr(\sigma_\mu^2\nabla^2\varphi), \quad Q_\mu(\varphi\otimes \varphi) = \alpha_\mu\Psi(\varphi\otimes \varphi)
\end{equation}
for some maps $b\colon\Pcal_w\times \R^d\to \R^d$, $\sigma\colon\Pcal_w\times \R^d\to \S^{d}$, and $\alpha\colon\Pcal_w\times \R^{2d}\to \R$. Here we use the notation $\Psi(\varphi\otimes \varphi)(x,y):=(\varphi(x)-\varphi(y))^2$ for all $\varphi\in C_w^\infty$ and $x\in \R$.

\begin{theorem}\label{thm2}
Assume that $b_\mu(x)=\widetilde b_{w\mu}(x)$, $\sigma_\mu(x)=\widetilde \sigma_{w\mu}(x)$, $\alpha_\mu(x,y)=\widetilde \alpha_{w\mu}(x,y)$ for some continuous maps $\widetilde b\colon\Xcal\times\R^d\to \R^d$, $\widetilde\sigma\colon\Xcal\times\R^d\to \S^{d}$ and $\widetilde \alpha\colon\Xcal\times\R^{2d}\to \R$. Assume that
\[
\frac{\widetilde b_{\nu,i}(x)}{1+|x|}\quad\text{and}\quad \frac{\widetilde \sigma_{\nu,ij}(x)^2}{1+|x|^2},\qquad i,j\in\{1,\ldots, d\},
\]
are continuous as functions from $\Xcal$ to $\R+C_0(\R^d)$, and that $\alpha_\mu(x,y)$ is  continuous as a function from $\Xcal$ to $\R+C_0(\R^{2d})$. Assume also that
\[
\sup_{x\in\R^d} \frac{|x|\,|\widetilde b_\nu(x)| +| \widetilde \sigma_\nu(x)|^2}{1+|x|^2} + \sup_{x,y\in \R^{d}}|\widetilde \alpha_\nu(x,y)|
\le c \langle 1,\nu\rangle^\gamma, \quad \nu\in\Xcal,
\]
for some constants $c,\gamma\ge0$. Then conditions \ref{T_Pw_1}--\ref{T_Pw_3_new} of Theorem~\ref{T_Pw} are satisfied, and thus there exists a possibly killed solution $X$ to the martingale problem for $(L,\Dcal_w,\Pcal_w)$ for every initial condition $\mu\in\Pcal_w$.
If one can take $\gamma=0$, then condition \ref{T_Pw_6} of Theorem~\ref{T_Pw} holds, and thus $X_t\in \Pcal_w$ for all $t\geq0$.
\end{theorem}

\begin{proof}
The proof  follows the proof of Theorem~\ref{thm1}, applying Theorem~\ref{T_optimality}\ref{T_optimality:2} to verify the positive maximum principle instead of Theorem~\ref{T_optimality}\ref{T_optimality:4}.
\end{proof}

\begin{example}
The Fleming--Viot diffusion was introduced by \cite{FV:79} and subsequently studied by many other authors. Its generator is of the form \eqref{eqn3} for $b\equiv0$, $\sigma$ constant, and $\alpha\equiv 1$. Here $E=\R$. A generalization of this model allows for so-called \emph{weighted sampling} by letting the coefficient $\alpha$ be non-constant. The interpretation is that the \emph{sampling-replacement rate} depends on the \emph{types} $x$ and $y$: $\alpha(x,y)$ denotes the rate at which an individual of type $x$ is replaced by one of type $y$. See Section~5.7.8 of \cite{D:93} for further details. Theorem~\ref{thm2} permits us to deal with the case where $\alpha$ in addition depends on $\mu$. Thus, the sampling-replacement rate depends not only on $x$ and $y$, but on the entire distribution $\mu$ of types. As a concrete example, choosing $w(x):=x^2$, we can let the sampling-replacement rate be monotonic in the variance of the distribution. Setting $\textup{Var}(\mu):=\langle (\fdot)^2,\mu\rangle-\langle (\fdot),\mu\rangle^2$, the corresponding operators are then given by
\begin{align*}
B_\mu(\varphi) = \frac 12\sigma^2\varphi'', \qquad Q_\mu(\varphi\otimes \varphi) = f\big(\textup{Var}(\mu)\big)\Psi(\varphi\otimes \varphi),
\end{align*}
for some nonnegative increasing bounded function $f\in C^\infty(\R)$.
\end{example}

Similarly to standard SDEs with linearly growing coefficients, the linear growth properties (implicit in Lemma~\ref{lem4}) imply that all moments of $\langle w,X_t\rangle$ are finite.

\begin{proposition}\label{prop21}
Let $L$ be as in Theorem~\ref{thm1} and assume it satisfies the assumptions given there for $\gamma=0$. Let $X$ be a solution to the martingale problem for $(L,\Dcal_w,\Pcal_w)$ for some initial condition $\overline\mu\in\Pcal_w$. Then the following conditions hold.
\begin{enumerate}
\item\label{bla1} $\E[\langle w, X_t\rangle^k]\leq \langle w, \overline\mu \rangle^k e^{C_kt}$ for all $k\in\N$, where $C_k:=k(k+1)C$ for some $C>0$. 
\item\label{bla2} $X$ solves the martingale problem for $(L,\Ecal_w,\Pcal_w)$ where $\Ecal_w$ denotes the algebra generated by $\{\mu\mapsto\langle\varphi,\mu\rangle\colon \varphi\in C_w^\infty\}$ and $Lf$ is given by \eqref{eq_Levy_type} for all $f\in\Ecal_w$. 
Moreover, $f(X) - f(X_0) - \int_0^\fdot Lf(X_s) ds$ is a $k$-integrable martingale for all $f\in\Ecal_w$.
\end{enumerate}
\end{proposition}

\begin{proof} To simplify the notation we just prove the case $d=1$.
Before to start observe that by the dominated convergence theorem the process
$
f(X) - f(\overline \mu) - \int_0^\fdot Lf(X_s) ds
$
denotes a bounded martingale for each map $f$ such that
\begin{equation} \label{eqn4}
\begin{minipage}[c][3em][c]{.8\textwidth}
\begin{center}
the pair $(f,Lf)$ with $Lf$ given by \eqref{eq_Levy_type} lies in\\ the bp-closure of the graph $\{(h,Lh)\colon h\in\Dcal_w\}$.
\end{center}
\end{minipage}
\end{equation}
We already shown in the proof of condition~\ref{T_Pw_6} during the proof of Theorem~\ref{thm1} that \eqref{eqn4} is satisfied by every function $f$ in the algebra generated by $\Dcal_w$ and $e^{-\langle w,\mu\rangle}$. Proceeding as in the proof of Lemma~\ref{lem3} we can extend this result to all  maps $f:\Pcal_w\to\R$ of the form
$
f(\mu):= p(f_1(\mu),\ldots,f_n(\mu))
$
for $p\in C^2(\R^n)$ with $p(0)=0$ and $f_1,\ldots,f_n$ being elements of that algebra.
\begin{enumerate}
\item[\ref{bla1}:] Fix now  $c>\langle w,\overline \mu\rangle$ and set $f_c(\mu):=p_c(\langle w,\mu\rangle) \langle w,\mu\rangle^k$ for $p_c(x)=p(x/c)$ where $p\in C^\infty(\R_+)$ satisfies
$1_{\{x\leq 1\}}\leq p(x)\leq1_{\{x\leq2\}}$. Note that $f_c$ satisfies \eqref{eqn4}.
Setting $T_c:=\inf\{t\geq0\ :\ \langle w,X_t\rangle\geq c\}$ we have, due to \eqref{eqn33a} for $\gamma=0$,
\begin{align*}
\E[f_c(X_{t\land T_c})]
&\leq\langle w,\overline \mu \rangle^k
+\int_0^tC_k\E[\langle w, X_{s}\rangle^k1_{\{s<T_c\}}]ds\\
&\leq\langle w, \overline\mu \rangle^k
+\int_0^tC_k\E[f_c( X_{s\land T_c})]ds.
\end{align*}
The Gronwall inequality implies that 
$\E[f_c(X_{t\land T_c})]\leq \langle w, \overline\mu \rangle e^{C_kt}$ and Fatou's lemma yields
\begin{align*}
\E[\langle w, X_t\rangle^k]
\leq\liminf_{c\to\infty}\E[p_c(\langle w, X_{t\land T_c}\rangle)\langle w, X_{t\land T_c}\rangle^k]
\leq \langle w, \overline\mu \rangle^k e^{C_kt}.
\end{align*}
\item[\ref{bla2}:] Fix now $g\in \Ecal_w$ and set $g_c(\mu):=p_c(\langle w,\mu\rangle)g(\mu)$.
Observe that each $g_c$ satisfies condition \eqref{eqn4}, $|g_c(\mu)|\leq C\langle w,\mu\rangle^k$, and $|Lg_c(\mu)|\leq C\langle w,\mu\rangle^k$ for some $k$ big enough and some constant $C$ not depending on $c$ and $\mu$. Since by \ref{bla1} we get 
$$
\E[\langle w,X_t\rangle^k - \langle w,\overline \mu\rangle^k - \int_0^t \langle w,X_s\rangle^k ds]
\leq \langle w, \overline\mu \rangle^k (e^{C_kt}+2+C_k^{-1}(e^{C_kt}-1))<\infty
$$
and $g_c\to g$ pointwise the claim follows  by the dominated convergence theorem. 
\end{enumerate}
This completes the proof.
\end{proof}

\section{McKean--Vlasov equations with common noise}\label{IIIsec42}

We continue to consider the setting of Section~\ref{three}: $E=\R^d$, $w(x)=|x|^p$ for $|x|>2$ and some $p\in(0,\infty)$, $L\colon\Dcal_w\to C(\Pcal_w)$ is a linear operator of L\'evy type \eqref{eq_Levy_type} with $\kappa=0$, $\Nu=0$, and $B$ and $Q$ given by
\begin{align*}
B_\mu(\varphi) &= b_\mu^\top\nabla\varphi+\frac 12\tr((\sigma_\mu^2+\tau_\mu^2)\nabla^2\varphi),\\
Q_\mu(\varphi\otimes \varphi) &= (\tau_\mu\nabla \varphi)\otimes (\tau_\mu\nabla \varphi),
\end{align*}
for some maps $b\colon\Pcal_w\times \R^d\to \R^d$ and $\sigma,\tau\colon\Pcal_w\times \R^d\to \S^{d}$.

\begin{definition}
A weak solution of the McKean--Vlasov equation specified by $(b,\sigma,\tau)$ is a tuple $(X,Z,W,W^0)$, defined on some filtered probability space, where $X$ and $Z$ are adapted with values in $\Pcal_w$ and $\R^d$, $W$ and $W^0$ are independent $d$-dimensional Brownian motions, and such that
\begin{equation}\label{eq_dZtisbttt}
d Z_t=b_{X_t}(Z_t)dt+\sigma_{X_t}(Z_t)dW_t+\tau_{X_t}(Z_t)dW_t^0
\end{equation}
and
\[
X_t=\P(Z_t\in\fdot\mid\Gcal_t)
\]
for some filtration $\G=(\Gcal_t)_{t\ge0}$ to which $W^0$ is adapted, and of which $W$ is independent.
\end{definition}

Our aim in this section is to give an existence result for the McKean--Vlasov equation specified by $(b,\sigma,\tau)$ by solving a martingale problem satisfied by the solution $(X,Z,W,W^0)$. The state space for this martingale problem is the product space $\Pcal_w\times\R^d\times\R^d\times\R^d$. We will show below that the solution satisfies
\begin{equation}\label{eq_dphiXty}
d\langle\varphi, X_t\rangle = \langle B_{X_t}\varphi, X_t\rangle dt + \langle \tau_{X_t} \nabla\varphi, X_t\rangle^\top dW_t^0
\end{equation}
for each $\varphi\in C_c^\infty(\R^d)$. The corresponding generator $H$ has domain $D(H)$ consisting of all algebraic combinations of functions $f(\mu)$, $\varphi(z)$, $\psi(x)$, $\theta(x^0)$ with $f\in\Dcal_w$ and $\varphi,\psi,\theta\in C^\infty_c(\R^d)$. In view of \eqref{eq_dZtisbttt} and \eqref{eq_dphiXty}, $H$ acts on functions of the form $h(\mu,z,x,x^0)=f(\mu)\varphi(z)\psi(x)\theta(x^0)$ by the somewhat cumbersome expression
\begin{align*}
Hh(\mu,z,x,x^0) &=  Lf(\mu)\varphi(z)\psi(x)\theta(x^0) + f(\mu) B_\mu\varphi(z)\psi(x)\theta(x^0) \\
&\quad + \frac12 f(\mu)\varphi(z)\Delta \psi(x) \theta(x^0) + \frac12 f(\mu)\varphi(z)\psi(x)\Delta \theta(x^0) \\
&\quad + \nabla\varphi(z)^\top\tau_\mu(z) \langle \tau_\mu \nabla(\partial f(\mu)), \mu\rangle\psi(x)\theta(x^0)  \\
&\quad + \nabla \theta(x^0)^\top\langle \tau_\mu \nabla(\partial f(\mu)), \mu\rangle\psi(x)\varphi(z) \\
&\quad + f(\mu)\theta(x^0)\nabla\varphi(z)^\top\sigma_\mu(z) \nabla\psi(x) \\
&\quad + f(\mu)\psi(x)\nabla\varphi(z)^\top\tau_\mu(z) \nabla\theta(x^0).
\end{align*}

\begin{theorem}\label{thm100}
Fix $\overline z\in \R^d$ and assume $b$, $\sigma$, $\tau$ satisfy the conditions of Theorem~\ref{thm1} for $\gamma=0$.  Then the martingale problem for $(H,D(H),\Pcal_w\times\R^d\times\R^d\times\R^d)$ with initial condition $(\delta_{\overline z},\overline z,0,0)$ has a solution $(X,Z, W, W^0)$, where $W$ and $W^0$ are independent $d$-dimensional Brownian motions. Moreover, the linear equation
\begin{equation}\label{eqn7}
\langle\varphi, Y_t\rangle=\varphi(\overline z)+\int_0^t\langle B_{X_s}\varphi, Y_s\rangle ds
+\int_0^t\langle \tau_{X_s} \nabla\varphi, Y_s\rangle^\top dW_s^0,\qquad \varphi\in C_c^\infty(\R^d).
\end{equation}
is satisfied for $Y=X$. If one has the compatibility conditions that $W$ is independent of the filtration $\G=(\Gcal_t)_{t\ge0}$ generated by $(X,W^0)$, and for all $s\le t$, $\Fcal_s$ and $\Gcal_t$ are conditionally independent given $\Gcal_s$, then \eqref{eqn7} is satisfied for $Y_t=\P(Z_t\in\fdot\mid\Gcal_t)$ as well. In particular, if in addition uniqueness holds for \eqref{eqn7}, then $(X,Z,W,W^0)$ is a weak solution of the McKean--Vlasov equation specified by $(b,\sigma,\tau)$.
\end{theorem}

The compatibility conditions on the filtrations $\F$ and $\G$ are rather implicit. However, similar conditions are known to be required elsewhere in the literature; see for instance page 114 in \cite{KX:99}, and the conditions of Theorem 2 in \cite{KSZ:78}. See also the remark at the beginning of page 142 in \cite{KSZ:78}. Let us also mention (without proof) that whenever $(X,W,W^0)$ solves the corresponding martingale problem, one can construct a process $\widetilde W$ such that $(X,\widetilde W,W^0)$ solves the same martingale problem and $\widetilde W$ is independent of the filtration $\G=(\Gcal_t)_{t\ge0}$ generated by $(X,W^0)$.

\begin{remark}
Let $f(\nu):=p(\langle \varphi_1,\nu\rangle,\ldots,\langle \varphi_n,\nu\rangle)$ for some nonnegative map $p\colon\R^n\to\R$ satisfying $p(0)=0$, some $\varphi_1,\ldots, \varphi_n\in C_c^\infty(\R^d)$, and all $\nu\in \Pcal_w$. Note that setting $\langle \varphi_i,\nu-\tilde\nu\rangle:=\langle \varphi_i,\nu\rangle-\langle\varphi_i,\tilde\nu\rangle$ we can naturally extend $f$ to  $\Pcal_w-\Pcal_w$. Consider now two solutions $Y$ and $\widetilde Y$ of \eqref{eqn7}. An application of It\^o's formula yields
$$\E[f(Y_t-\widetilde Y_t)]=\int_0^t\E[L_{X_s}f(Y_s-\widetilde Y_s)]ds,$$
where $L_\mu f(\nu)= \langle B_\mu(\partial f(\nu)), \nu\rangle + \frac12 \langle Q_\mu(\partial^2f(\nu)), \nu^2\rangle$. If $f$ additionally satisfies  
\begin{equation}\label{eqn300}
\text{$|L_\mu f(\nu)|\leq C f(\nu)$ for all $\nu\in \Pcal_w-\Pcal_w$ and $\mu\in \Pcal_w$},
\end{equation}
an application of the Gronwall inequality yields $\E[f(Y_t-\widetilde Y_t)]=0$, and thus that $Y_t-\widetilde Y_t\in\{f=0\}$. If this condition holds for sufficiently many $f$, we would be able to conclude that $Y_t=\widetilde Y_t$ almost surely and that uniqueness holds for \eqref{eqn7}. We illustrate a situation where this is the case in the following example.

Let $d=1$, $\overline z\in[0,1]$, and assume that $Y_t([0,1])=\widetilde Y_t([0,1])=1$ for each $t\geq0$.
 Assume that the maps $x\mapsto b_\mu(x) $ and $x\mapsto\tau_\mu(x)$ are polynomials of degree at most 1 and the map $x\mapsto\sigma_\mu(x)^2$ is a polynomial of degree at most 2.
This in particular implies that $B_\mu$ and $Q_\mu$ are polynomial operators in the sense of \cite{FL:17}, meaning that they map any polynomial to a polynomial of the same or lower degree.
Fix then $H_0,\ldots,H_{m}\in C_c^\infty(\R)$ such that $H_i(x)=x^i$ for each $x\in[0,1]$ and set $p_m(\nu)= \sum_{i=0}^{m}\langle H_i,\nu\rangle^2$. Note that for each $\nu\in \Pcal_w-\Pcal_w$ such that $\supp(\nu)\subseteq[0,1]$ we have
\begin{align*}
|L_\mu p_m(\nu)|&=
|\sum_{i=0}^m2\langle H_i,\nu\rangle\langle B_\mu H_i,\nu\rangle+\langle Q_\mu(H_i\otimes H_i),\nu^2\rangle|\\
&=|\sum_{i,j=0}^{m}\alpha^\mu_{ij}\langle H_i,\nu\rangle\langle H_j,\nu\rangle|\\
&\leq (m+1)\sup_{ij}|\alpha^\mu_{ij}|p_m(\nu),
\end{align*}
for some $\alpha^\mu_{ij}\in \R$.
This implies that if $\sup_{\mu\in \Pcal_w} |\alpha^\mu_{ij}|<\infty$, then condition \eqref{eqn300} is satisfied, and $\langle H_i,Y_t\rangle=\langle H_i, \widetilde Y_t\rangle$ for each $i\in\{1,\ldots,m\}$. Since $m$ was arbitrary the  same conclusion holds for each $i\in \N$. Since two measures on $[0,1]$ have the same moments if and only if they are the same, it follows that $Y_t=\widetilde Y_t$ almost surely and uniqueness holds for \eqref{eqn7}.
\end{remark}

The rest of the section is devoted to the proof of Theorem~\ref{thm100}. We start with a corollary of Theorem~\ref{thm1}.

\begin{corollary}\label{cor2}
Assume that $b$, $\sigma$, and $\tau$ satisfy the conditions of Theorem~\ref{thm1} for $\gamma=0$. Then  there exists a solution $(X,Z,W,W^0)$ to the martingale problem for $(H,D(H),\Pcal_w\times\R^d\times\R^d\times\R^d)$ for every initial condition $(\mu,z,x,x^0)\in\Pcal_w\times\R^d\times\R^d\times\R^d$.
\end{corollary}
\begin{proof}
We first observe that $H$ satisfies the positive maximum principle on $\Pcal_w\times \R^d\times \R^d\times \R^d$. This can be proven by the classical optimality conditions on $\R^{d}\times\R^d\times \R^d$, Theorem~\ref{T_optimality}\ref{T_optimality:1}, and a slightly modification of the argument in the proof of Theorem~\ref{T_optimality}\ref{T_optimality:3b}.

Observe then that the concepts introduced in Section~\ref{main} can be generalized by setting
$$\Tcal:\Pcal_w\times \R^{d}\times\R^d\times \R^d\to \Xcal\times \R^{d}\times\R^d\times \R^d,\qquad \Tcal(\mu,z,x,x^0):=(T(\mu),z,x,x^0)$$
and calling $f\colon\Pcal_w\times\R^{d}\times\R^d\times \R^d\to\R$  of  $C_0$ type if $f\circ \Tcal^{-1}\colon \Tcal(\Pcal_w\times \R^{d}\times\R^d\times \R^d)\to\R$ extends to a $C_0$ function on $\Xcal\times \R^{d}\times\R^d\times \R^d$. Since $L$ satisfies the conditions of  Theorem~\ref{thm1} we know that $H h$ is of $C_0$ type for every $h\in D(H)$. Following the proof of Theorem~\ref{T_Pw} we can conclude that there exists a possibly killed solution of the martingale problem for $(\widetilde H, D(H)\circ \Tcal^{-1},\Xcal\times \R^{d}\times\R^d\times \R^d)$ where
$$\widetilde H\widetilde h = H(\widetilde h\circ \Tcal)\circ \Tcal^{-1}.$$
Using that by Theorem~\ref{thm1} the operator $L$ satisfies conditions  \ref{T_Pw_3_new}--\ref{T_Pw_6} of Theorem~\ref{T_Pw}, we can conclude the proof by following the proof of Theorem~\ref{T_Pw}.
\end{proof}
The fact that a solution of the martingale problem is also a weak solution of the corresponding SDE is due, in the classical case, to  \cite{SV:72}.

\begin{lemma}\label{L_s6nbr35a05yfgd}
Assume that the conditions of Corollary~\ref{cor2} are satisfied and consider a solution $(X,Z, W, W^0)$  to the martingale problem for $(H,D(H),\Pcal_w\times\R^d\times\R^d\times\R^d)$ with initial condition $(\delta_{\overline z},\overline z,0,0)$. Then $W,W^0$ are independent Brownian motions, and \eqref{eq_dZtisbttt} and \eqref{eq_dphiXty} hold for each $\varphi\in C_c^\infty(\R^d)$.
\end{lemma}

\begin{proof}
For $h(\mu,z,x,x^0)=\psi(x,x^0)$, we have that $Hh(\mu,z,x,x^0)=\frac 1 2 \Delta\psi(x,x^0)$ is the Laplacian. Thus $W$ and $W^0$ are independent Brownian motions. To prove \eqref{eq_dphiXty}, we must show that the process $\langle\varphi, X_t\rangle-\int_0^t\langle B_{X_s}\varphi, X_s\rangle ds-\int_0^t\langle \tau_{X_s} \nabla\varphi(X_s),X_s\rangle^\top dW_s^0$, which is known to be a martingale due to Proposition~\ref{prop21}, is constant. This is done by verifying that its quadratic variation is zero; we omit the details. The proof of \eqref{eq_dZtisbttt} is similar.
\end{proof}

\begin{lemma}\label{L_s6nbr3asdgsfghsdfgyfgd}
Assume that the conditions of Corollary~\ref{cor2} are satisfied and consider a solution $(X,Z, W, W^0)$ to the martingale problem for $(H,D(H),\Pcal_w\times\R^d\times\R^d\times\R^d)$ with initial condition $(\delta_{\overline z},\overline z,0,0)$. If one has the compatibility conditions that $W$ is independent of the filtration $\G=(\Gcal_t)_{t\ge0}$ generated by $(X,W^0)$, and for all $s\le t$, $\Fcal_s$ and $\Gcal_t$ are conditionally independent given $\Gcal_s$, then the conditional law process $Y_t:=\P( Z_t\in\fdot\mid\Gcal_t)$ satisfies \eqref{eqn7}.
\end{lemma}

Before we start the proof, observe that condition \eqref{eqn33} for $\gamma=0$ implies
\begin{equation}\label{eqn6}
\sup_{x\in\R,\ \mu\in \Pcal_w} { | B_\mu(x)\varphi(x)|+| \tau_\mu(x)\nabla\varphi(x)|}
<\infty
\end{equation}
for each $ \varphi\in C_c^\infty(\R^d)$.

\begin{proof}
Observe that an application of the It\^o formula yields
\begin{align*}
\varphi( Z_t)&=\varphi(\overline z)+\int_0^t\!B_{X_s}\varphi( Z_s)ds+\int_0^t(\sigma_{X_s}( Z_s)\nabla \varphi( Z_s))^\top dW_s+\int_0^t(\tau_{X_s}( Z_s)\nabla \varphi( Z_s))^\top dW_s^0,
\end{align*}
for each $\varphi\in C_c^\infty(\R^{d})$. 
Note that condition \eqref{eqn6} yields that 
$$\E[|\tau_{X_s}( Z_s)\nabla \varphi( Z_s)|^2]\text{ and }\E[|\sigma_{X_s}( Z_s)\nabla \varphi( Z_s)|^2]$$ are almost surely bounded in $s$ for each $\varphi\in C_c^\infty(\R^{d})$. Combining Lemma~\ref{lemA},  Fubini theorem for conditional expectation, and the $\Gcal_s$-measurability of $X_s$ we then get
\begin{align*}
\E[\varphi(  Z_t)\mid\Gcal_t]&=\varphi(\overline z)+\int_0^t\E[ B_{X_s}(  Z_s)\mid\Gcal_s]ds+\int_0^t\E[(\tau_{X_s}(  Z_s)\nabla \varphi( Z_s))\mid\Gcal_s]^\top dW_s^0\\
&=\varphi(\overline z)+\int_0^t\langle B_{X_s}\varphi, Y_s\rangle ds
+\int_0^t\langle \tau_{X_s} \nabla\varphi, Y_s\rangle^\top dW_s^0,
\end{align*}
for each $\varphi\in C^\infty_c(\R^d)$. 
\end{proof}

\begin{proof}[Proof of Theorem~\ref{thm100}]
The result follows directly from Corollary~\ref{cor2} and Lemmas~\ref{L_s6nbr35a05yfgd} and~\ref{L_s6nbr3asdgsfghsdfgyfgd}.
\end{proof}

\appendix
\section{A Fubini type result}
The result presented in this section is based on Theorem~2 in \cite{KSZ:78} and its proof.
Let $(\Omega, \Fcal,\F=(\Fcal_t)_{t\geq0},\P)$ be a filtered probability space endowed with two $d$-dimensional Brownian motions $W$ and $W^0$. Consider then a second filtration $\G=(\Gcal_t)_{t\ge0}$ to which $W^0$ is adapted, and of which $W$ is independent. 
\begin{lemma}\label{lemA}
If $\Fcal_s$ and $\Gcal_t$ are conditionally independent given $\Gcal_s$, then
$$\E[\int_0^t H_s^\top dW_s^0\mid\Gcal_t]=\int_0^t \E [H_s^\top \mid\Gcal_s]dW_s^0\qquad\text{and}\qquad \E[\int_0^t H_s^\top dW_s\mid\Gcal_t]=0,$$
for each square integrable continuous process $H$ satisfying $\int_0^t\E[|H_s|^2]ds<\infty$.
\end{lemma}
Observe that the conditional independence assumption implies that each $\G$-mar\-tin\-gale  is also an $\F$-martingale. This condition 
is automatically satisfied if $\Gcal_t:=\sigma(W_s^0,s\leq t)$.

\begin{proof}
Set $\xi_t:=\E[\int_0^tH_s^\top dW^0_s\mid\Gcal_t]$ and note that  $\xi$ and $W_0$ are square integrable martingales with respect to $\G$ and thus with respect to $\F$. Moreover, since $W_t^0\xi_t=\E[W^0_t\int_0^tH_s^\top dW^0_s\mid\Gcal_t]$, for each $A\in\Gcal_u$ we can compute
\begin{align*}
\E[(W^0_t\xi_t-W^0_u\xi_u)1_A]&=\E\Big[\Big(W^0_t\E[\int_0^tH_s^\top dW^0_s\mid\Gcal_t] 
-W^0_u\E[\int_0^uH_s^\top dW^0_s\mid\Gcal_u] \Big)1_A\Big]\\
&=\E\Big[\Big(W^0_t\int_0^tH_s^\top dW^0_s-W^0_u\int_0^uH_s^\top dW^0_s\Big)1_A\Big]\\
&=\E[\int_u^t(H_s^\top 1)ds1_A]\\
&=\E\Big[\int_u^t\E[H_s\mid\Gcal_s]^\top 1ds1_A\Big],
\end{align*}
and thus conclude that $W_t^0\xi_t-\int_0^t\E[H_s\mid\Gcal_s]^\top 1ds$ is an $\F$-martingale.
This in particular implies that $\int_0^t\E[H_s\mid\Gcal_s]^\top 1ds$ is the predictable quadratic covariation of $\xi $ and $W^0$ with respect to $\F$. Since $\int_0^tH_s^\top dW^0_s$ is a square integrable $\F$-martingale and  $\xi_t^2=\E[\xi_t\int_0^tH_s^\top dW^0_s\mid\Gcal_t]$ we get
\begin{align*}
\E[\xi_t^2]&=\E[\xi_t\E[\int_0^tH_s^\top dW^0_s\mid\Gcal_t] ]\\
&=\E[\xi_t\int_0^tH_s^\top dW^0_s]\\
&=\E[\int_0^tH_s^\top \E[H_s\mid\Gcal_s]ds]\\
&=\E[\int_0^t\E[H_s\mid\Gcal_s]^\top \E[H_s\mid\Gcal_s]ds].
\end{align*}
Similarly we also get that $\E[\xi_t\int_0^t\E[H_s\mid\Gcal_s]^\top dW_s^0]=\E[\int_0^t\E[H_s\mid\Gcal_s]^\top \E[H_s\mid\Gcal_s]ds]$. 
Using that $\E[(\int_0^t\E[H_s\mid\Gcal_s]^\top dW_s^0)^2]=\E[\int_0^t\E[H_s\mid\Gcal_s]^\top \E[H_s\mid\Gcal_s]ds]$ we can thus conclude that
$$\E[(\xi_t-\int_0^t\E[H_s\mid\Gcal_s]^\top dW_s^0)^2]=0$$
proving that $\E[\int_0^tH_s^\top dW^0_s\mid\Gcal_t]=\int_0^t\E[H_s\mid\Gcal_s]^\top dW_s^0$.

For the second part set $\eta_t:=\E[\int_0^tH_s^\top dW_s\mid\Gcal_t]$. Since $\eta$ and $W$ are two independent continuous $\F$-martingales we already know that $(\eta_t W_t)_{t\geq0}$ defines a square integrable $\F$-martingale. Proceeding as in the first part we can thus conclude that
$\E[\eta_t^2]
=\E[\eta_t\int_0^tH_s^\top dW_t]=0$.
\end{proof}



\begin{thebibliography}{28}
\providecommand{\natexlab}[1]{#1}
\providecommand{\url}[1]{\texttt{#1}}
\expandafter\ifx\csname urlstyle\endcsname\relax
  \providecommand{\doi}[1]{doi: #1}\else
  \providecommand{\doi}{doi: \begingroup \urlstyle{rm}\Url}\fi

\bibitem[Abi~Jaber et~al.(2019)Abi~Jaber, Bouchard, and Illand]{ABI:19}
E.~Abi~Jaber, B.~Bouchard, and C.~Illand.
\newblock Stochastic invariance of closed sets with non-lipschitz coefficients.
\newblock \emph{Stochastic Processes and their Applications}, 129\penalty0
  (5):\penalty0 1726 -- 1748, 2019.

\bibitem[Carmona and Delarue(2017)]{CD:17}
R.~Carmona and F.~Delarue.
\newblock \emph{{Probabilistic Theory of Mean Field Games with
  Applications~I-II}}.
\newblock Springer, 2017.

\bibitem[Cuchiero(2019)]{C:19}
C.~Cuchiero.
\newblock Polynomial processes in stochastic portfolio theory.
\newblock \emph{Stochastic processes and their applications}, 129\penalty0
  (5):\penalty0 1829--1872, 2019.

\bibitem[Cuchiero et~al.(2019)Cuchiero, Larsson, and Svaluto-Ferro]{CLS:18}
C.~Cuchiero, M.~Larsson, and S.~Svaluto-Ferro.
\newblock Probability measure-valued polynomial diffusions.
\newblock \emph{Electronic Journal of Probability}, 24, 2019.

\bibitem[Da~Prato and Frankowska(2004)]{DF:04}
G.~Da~Prato and H.~Frankowska.
\newblock Invariance of stochastic control systems with deterministic
  arguments.
\newblock \emph{Journal of Differential Equations}, 200\penalty0 (1):\penalty0
  18 -- 52, 2004.

\bibitem[Dawson(1977)]{D:77}
D.~Dawson.
\newblock The critical measure diffusion process.
\newblock \emph{Zeitschrift f{\"u}r Wahrscheinlichkeitstheorie und Verwandte
  Gebiete}, 40\penalty0 (2):\penalty0 125--145, 1977.

\bibitem[Dawson(1978)]{D:78}
D.~Dawson.
\newblock Geostochastic calculus.
\newblock \emph{The Canadian Journal of Statistics / La Revue Canadienne de
  Statistique}, 6\penalty0 (2):\penalty0 143--168, 1978.

\bibitem[Dawson(1993)]{D:93}
D.~Dawson.
\newblock Measure-valued {M}arkov processes.
\newblock In \emph{\'{E}cole d'\'{E}t\'e de {P}robabilit\'es de {S}aint-{F}lour
  {XXI}--1991}, Lecture Notes in Mathematics, pages 1--260. Springer, 1993.

\bibitem[Dawson and Vaillancourt(1995)]{DV:95}
D.~Dawson and J.~Vaillancourt.
\newblock {Stochastic McKean-Vlasov equations}.
\newblock \emph{Nonlinear Differential Equations and Applications NoDEA},
  2\penalty0 (2):\penalty0 199--229, 1995.

\bibitem[Etheridge(2011)]{E:11}
A.~Etheridge.
\newblock {Some Mathematical Models from Population Genetics}.
\newblock In \emph{\'{E}cole d'\'{E}t\'e de {P}robabilit\'es de {S}aint-{F}lour
  {XXXIX-2009}}, Lecture Notes in Mathematics. Springer, 2011.

\bibitem[Ethier and Kurtz(1987)]{EK:87}
S.~N. Ethier and T.~G. Kurtz.
\newblock \emph{The Infinitely-Many-Alleles Model with Selection as a
  Measure-Valued Diffusion}, pages 72--86.
\newblock Springer, Berlin, 1987.

\bibitem[Ethier and Kurtz(1993)]{EK:93}
S.~N. Ethier and T.~G. Kurtz.
\newblock Fleming--{V}iot processes in population genetics.
\newblock \emph{SIAM Journal on Control and Optimization}, 31\penalty0
  (2):\penalty0 345--386, 1993.

\bibitem[Ethier and Kurtz(2005)]{EK:05}
S.~N. Ethier and T.~G. Kurtz.
\newblock \emph{Markov Processes: Characterization and Convergence}.
\newblock Wiley Series in Probability and Statistics. Wiley, 2 edition, 2005.

\bibitem[Fernholz(2002)]{F:02}
R.~Fernholz.
\newblock \emph{Stochastic Portfolio Theory}.
\newblock Applications of Mathematics. Springer-Verlag, New York, 2002.

\bibitem[Fernholz and Karatzas(2009)]{FK:09}
R.~Fernholz and I.~Karatzas.
\newblock {Stochastic portfolio theory: an overview.}
\newblock \emph{Handbook of numerical analysis}, 15:\penalty0 89--167, 2009.

\bibitem[Filipovi{\'c} and Larsson(2020)]{FL:17}
D.~Filipovi{\'c} and M.~Larsson.
\newblock {Polynomial Jump-Diffusion Models}.
\newblock \emph{Stochastic Systems}, 10\penalty0 (1):\penalty0 71--97, 2020.

\bibitem[Fleming and Viot(1979)]{FV:79}
W.~H. Fleming and M.~Viot.
\newblock Some measure-valued {M}arkov processes in population genetics theory.
\newblock \emph{Indiana Univ. Math. J.}, 28\penalty0 (5):\penalty0 817--843,
  1979.

\bibitem[Florchinger and Le~Gland(1992)]{FL:92}
P.~Florchinger and F.~Le~Gland.
\newblock Particle approximation for first order stochastic partial
  differential equations.
\newblock In \emph{Applied stochastic analysis}, pages 121--133. Springer,
  1992.

\bibitem[F{\"o}llmer and Schied(2004)]{FS:04}
H.~F{\"o}llmer and A.~Schied.
\newblock \emph{{Stochastic Finance: An Introduction in Discrete Time}}.
\newblock De Gruyter studies in mathematics. Walter de Gruyter, 2 edition,
  2004.

\bibitem[Huang(1987)]{HS:87}
K.~Huang.
\newblock \emph{Statistical mechanics}.
\newblock Wiley, 1987.

\bibitem[Kailath et~al.(1978)Kailath, Segall, and Zakai]{KSZ:78}
T.~Kailath, A.~Segall, and M.~Zakai.
\newblock {Fubini-Type Theorems for Stochastic Integrals}.
\newblock \emph{Sankhy{\=a}: The Indian Journal of Statistics, Series A
  (1961-2002)}, 40\penalty0 (2):\penalty0 138--143, 1978.

\bibitem[Klenke(2013)]{K:13}
A.~Klenke.
\newblock \emph{{Probability Theory: A Comprehensive Course}}.
\newblock Universitext. Springer London, 2 edition, 2013.

\bibitem[Kurtz and Xiong(1999)]{KX:99}
T.~Kurtz and J.~Xiong.
\newblock {Particle representations for a class of nonlinear SPDEs}.
\newblock \emph{Stochastic Processes and their Applications}, 83, 1999.

\bibitem[Perkins(2002)]{P:02}
E.~Perkins.
\newblock {Dawson-Watanabe Superprocesses and Measure-valued Diffusions}.
\newblock In \emph{{{\`E}cole d'{\`E}t{\`e} de Probabilit{\`e}s de Saint-Flour
  XXIX--1999}}, Lecture Notes in Mathematics, pages 125--329. Springer, 2002.

\bibitem[Stroock and Varadhan(1972)]{SV:72}
D.~W. Stroock and S.~R.~S. Varadhan.
\newblock {On the support of diffusion processes with applications to the
  strong maximum principle}.
\newblock In \emph{Proceedings of the Sixth Berkeley Symposium on Mathematical
  Statistics and Probability (Univ. California, Berkeley, Calif., 1970/1971)},
  volume~3, pages 333--359, 1972.

\bibitem[Sznitman(1991)]{S:91}
A.~S. Sznitman.
\newblock {Topics in Propagation of Chaos}.
\newblock In \emph{{{\'E}cole d'{\'E}t{\'e} de Probabilit{\'e}s de Saint-Flour
  XIX--1989}}, Lecture Notes in Mathematics, pages 165--251. Springer, 1991.

\bibitem[Villani(2008)]{V:08}
C.~Villani.
\newblock \emph{Optimal transport: old and new}, volume 338.
\newblock Springer Science \& Business Media, 2008.

\bibitem[Watanabe(1968)]{W:68}
S.~Watanabe.
\newblock A limit theorem of branching processes and continuous state branching
  processes.
\newblock \emph{J. Math. Kyoto Univ.}, 8\penalty0 (1):\penalty0 141--167, 1968.

\end{thebibliography}
\end{document}